\documentclass[a4paper,11pt]{article}

\usepackage{amsfonts}
\usepackage{amssymb}
\usepackage{amsthm}
\usepackage{amsmath}
\usepackage{a4wide}
\usepackage{mathrsfs}
\usepackage{epsfig}
\usepackage{palatino}
\usepackage{tikz,fp,ifthen}
\usepackage{float}

\newcommand{\pdfgraphics}{\ifpdf\DeclareGraphicsExtensions{.pdf,.jpg}\else\fi}
\usepackage{graphicx}

\usepackage{color}
\definecolor{citegreen}{rgb}{0,0.6,0}
\definecolor{refred}{rgb}{0.8,0,0}
\usepackage[colorlinks, citecolor=citegreen,linkcolor=refred]{hyperref}

\usepackage{mathtools}
\mathtoolsset{showonlyrefs}

\numberwithin{equation}{section}

\theoremstyle{plain}

\newtheorem*{theorem*}{Theorem}

\newtheorem{teo}{Theorem}[section]
\newtheorem{lemma}[teo]{Lemma}
\newtheorem{prop}[teo]{Proposition}

\theoremstyle{definition}
\newtheorem{dfnz}[teo]{Definition}

\newtheorem{prob}[teo]{Problem}

\theoremstyle{remark}
\newtheorem{rem}[teo]{Remark}

\numberwithin{equation}{section}

\renewcommand{\t }{\tau }

\newcommand{\intbar}{\etaathop{\int\etaakebox(-13.5,0){\rule[4pt]{.7em}{0.3pt}}
		\kern-6pt}\nolimits}

\newcommand{\be}{\begin{equation}}
\newcommand{\ee}{\end{equation}}
\newcommand{\bea}{\begin{equation*}}
\newcommand{\eea}{\end{equation*}}

\def\be{\begin{equation}}
\def\ee{\end{equation}}
\def\bea{\begin{eqnarray*}}
	\def\bean{\begin{eqnarray}}
	\def\eean{\end{eqnarray}}
	\def\eea{\end{eqnarray*}}

\begin{document}
\pdfgraphics 

\title{Willmore flow of planar networks}

\author{Harald Garcke 
\footnote{Fakult\"at f\"ur Mathematik, Universit\"at Regensburg, Universit\"atsstrasse 31, 
93053 Regensburg, Germany}
 \and Julia Menzel
\footnotemark[1]
\and Alessandra Pluda
\footnotemark[1]
}

\maketitle

\begin{abstract}
\noindent  Geometric gradient flows for elastic energies of Willmore
type play an important role in mathematics and in many applications.
The evolution of elastic curves has been studied in detail both for
closed as well as for open curves. Although elastic flows for networks
also have many interesting features, they have not been studied so far
from the point of view of mathematical analysis. So far it was not even 
clear what are appropriate boundary conditions at junctions. In this paper we
give a well-posedness result for Willmore flow of networks in different
geometric settings and hence lay a foundation for further mathematical
analysis. A main point in the proof is to check whether different 
proposed boundary
conditions lead to a well posed problem. In this context one has to
check the  Lopatinskii--Shapiro condition in order to apply
the Solonnikov theory for linear parabolic systems in H\"older spaces
which is needed in a fixed point argument. We also show that 
the solution we get is unique in a purely geometric sense.

\end{abstract}



\section{Introduction}
A planar network $\mathcal{N}$ is a connected set composed 
of a finite number of sufficiently smooth curves $\mathcal{N} ^i$ that meet at junctions. 
We consider two types of networks of three curves: Theta--networks and Triods. 
The three regular curves of a Theta--network intersect each other at their endpoints in two triple junctions.
A Triod is composed by three regular curves that intersect each other at a triple junction 
and have the other three endpoints on the boundary of an 
open  set $\Omega\subset\mathbb{R}^2$.

\begin{figure}[H]
\begin{center}
\begin{tikzpicture}[scale=1]
\draw[shift={(0,0)}] 
(-1.73,-1.3) 
to[out= 180,in=160, looseness=1] (-2.8,0) 
to[out= 140,in=150, looseness=1.5] (-1.5,1) 
(-2.8,0)
to[out=-30,in=90, looseness=0.9] (-1.25,-0.75)
(-1.5,1)
to[out= -30,in=90, looseness=0.9] (-1,0)
to[out= -90,in=10, looseness=0.9] (-1.25,-0.75)
to[out= -140,in=0, looseness=0.9](-1.73,-1.3);
\path[shift={(0,0)}] 
(-1.35,-0.95)node[right]{$O^1$}
 (-1.5,0.05)[left] node{$\gamma^2$}
 (-0.6,.9)[left] node{$\gamma^1$}
 (-2.6,-0.95)[left] node{$\gamma^3$}
 (-3.2,0.35) node[below] {$O^2$};
 \draw[white]
 (-3,-1.5)--(-1,-1.5);
 \end{tikzpicture}\qquad
 \begin{tikzpicture}[scale=0.6]
\draw[shift={(-2,0)}]
(-3.73,0) node[left]{$P^1$}
to[out= 50,in=-150, looseness=1] (-2,0)
to[out= 60,in=180, looseness=1.5] (-0.45,1.55)
(-2,0)
to[out= -60,in=180, looseness=0.9] (-0.75,-1.75);
\draw[color=black,scale=1,domain=-3.141: 3.141,
smooth,variable=\t,shift={(-3.72,0)},rotate=0]plot({2.*sin(\t r)},
{2.*cos(\t r)});
\path[font=\large,shift={(-2,0)}]
    (0.8,0) node[below] {$\Omega$}
    (-3,0.13) node[below] {$\gamma^1$}
    (-1.5,1) node[right] {$\gamma^3$}
    (-0.7,-1)[left] node{$\gamma^2$}
    (-2.2,0.75) node[below] {$O$}
    (-0.21,1.35)node[above]{$\,\,\,\,\,\, P^3$}
    (-0.55,-1.65) node[below] {$\,\,\,\, P^2$};
\end{tikzpicture}\quad\quad
\end{center}
\begin{caption}{A Theta--network and a Triod}
\end{caption}
\end{figure}
We consider the elastic energy of the network, penalizing its global length, that is
$$
E_\mu(\mathcal{N}):=\sum_{i=1}^3\int_{\mathcal{N}^i}\left(\left( k^i\right)^2+\mu\right)\mathrm{d}s^i\,,\quad\quad\mu>0\,,$$
where $s^i$ denotes the arc length parameter of the curve $\mathcal{N}^i$
and $k^i$ the curvature.
We are interested in the $L^2$--gradient flow of the energy $E_\mu$:
we start with an initial network (Triod or Theta) and let it evolve with a normal velocity that induces the steepest descent of the energy with respect to the $L^2$--inner product.
The curves need to stay attached during the evolution but the junctions are allowed to move. 
\medskip

Our result is inspired by several others in the case of curves~\cite{lin2, dalpoz, dapozspe, dziukkuwertsch}.
The problem for networks was first proposed by Barrett, Garcke and N\"{u}rnberg: 
in~\cite{bargarnu} 
 it is shown that if $\left(\mathcal{N}(t)\right)_{t\in[0,T]}$ is 
a time dependent family of networks evolving according to the $L^2$--gradient flow of $E_\mu$, then 
the normal velocity of each of its curves is 
$$
-2k_{ss}-k^3+\mu k\,.
$$
Depending on which constraints one imposes at the endpoints of the curves one obtains a different flow. 
We consider two different types of properties at the triple junctions.
In one case we only demand that the curves stay attached during the evolution 
which we call $C^0$--flow. 
In the other case we additionally impose an angle condition at the triple junction(s). We require all angles to be equal and name this evolution $C^1$--flow.
We note that there are several variants  in the case of a Triod since different boundary conditions are possible at the endpoints on $\partial\Omega$, 
namely Navier (fixed endpoints and zero curvature) and clamped 
(position and unit tangent vectors are prescribed).
Depending on whether we consider the $C^0$-- or the $C^1$--flow
different natural boundary conditions have to be prescribed at the
triple junction(s). These follow from computing the first variation of
the energy taking the constraints at the junction into account. The main
question we would like to answer is whether the resulting boundary 
conditions at the triple junction(s) lead to a well-posed evolution
problem. The answer will be positive with the exception of one
degenerate geometric situation in  the $C^0$--case.

We are looking for classical solutions to the $C^0$-- and $C^1$--flow:
we require that
each evolving curve $\mathcal{N}^i$ of the network $\mathcal{N}$ admits a time dependent parametrisation
$$
\gamma^i\in C^{\frac{4+\alpha}{4},4+\alpha}([0,T]\times[0,1])\,.
$$ In particular the initial network needs to fulfil certain compatibility conditions to be admissible. These will be specified in 
Definitions~\ref{admg} and~\ref{admg1}.

\medskip

The result of this paper is the following:
\begin{theorem*}
Let $\mathcal{N}_0$ be a geometrically admissible initial network for the $C^0$--flow (or the $C^1$--flow). Then there exists a positive time $T$ such that within the interval $[0,T]$ the $C^0$--flow ($C^1$--flow) admits a unique solution 
(as precisely defined in~\ref{geometricsolution} and~\ref{geometricsolution1}).
\end{theorem*}

We give a detailed proof in the case of the $C^0$--flow for a Theta--network. 
In Section~\ref{C1flow} we show how to adapt the proof to Triods moving
according to the $C^1$--flow. Combining the techniques used to prove the 
two results we indeed cover all possible situations.

\medskip

To prove existence of the $C^0$--flow we translate the geometric problem to the level of parame--trisations. 
As roughly explained before the notion of $L^2$--gradient flow gives
rise to a boundary value problem for the parametrisations. As the motion
in tangential direction is not specified in the formulation of this
geometric problem, one needs to fix a tangential velocity in order to
obtain a well posed parabolic PDE for the parametrisation. Moreover  we have one tangential degree of freedom for each curve at each triple junction which forces us to fix another boundary condition. This has to be chosen in such a way that the geometric problem does not change! It turns out that it is convenient to demand
$$
\langle \gamma^i_{xx},\tau^i\rangle =0
$$
at both triple junctions and at every time
which we call tangential second order condition. 
A 
key point is to ensure that solving the analytic problem is enough to obtain a unique solution to the original problem formulated for networks as geometric objects. This is shown in Theorem~\ref{geomexistence} following the approach presented in~\cite{garckenov}.
This justifies to concentrate on the analytic problem which is a matter 
of solving the system~\eqref{thetaC0} of parabolic quasilinear PDEs of fourth order with coupled boundary conditions. 
To this end we study a linearised problem and deduce existence of the system from the linearised problem by a contraction argument. Our linearised problem~\eqref{lyntheta} is derived in~\ref{linearisation} and solved in~\ref{Solo}. 
As the boundary conditions are of different order, the theory of Solonnikov~\cite{solonnikov2} is the appropriate one to show existence of the linearised problem in H\"{o}lder spaces. The system is parabolic in the sense of Solonnikov. To apply~\cite[Theorem 4.9]{solonnikov2} we notice that   
the Lopatinskii--Shapiro condition is satisfied
in the case one completes the under determined boundary value problem by the tangential second order condition.
We conclude the proof 
in Section~\ref{contractionargu} giving the contraction argument.

\medskip

The first short time existence result for the elastic flow 
of curves was obtained by Polden~\cite{poldenthesis, polden}. The author
considered closed curves in the plane and also showed global existence. 
Afterwards the result was extended to curves in $\mathbb{R}^n$ in~\cite{dziukkuwertsch},
where also global existence is proved. For other results in a similar framework 
we also refer to~\cite{koiso, langersinger, wen}. More recently the elastic flow of open curves in $\mathbb{R}^n$ has been studied. 
Especially the long time behaviour has been investigated for various boundary conditions~\cite{lin2, dalpoz,  dapozspe, lin1, okabe}. In~\cite{adrian}
short time existence of the elastic flow of open clamped curves is shown.

We emphasize that in the case of networks we are concerned with a system of equations with coupled boundary conditions. Moreover the triple junction is allowed to move which gives rise to tangential degrees of freedom. One carefully has to study the problem in order to choose meaningful extra conditions at the boundary without changing the geometric flow.

\medskip

We notice that all results are valid not only for $\mu>0$ but also in the case $\mu\in\mathbb{R}$.  This is due to the fact that the linearisation and the existence of the linearised system only depend on the highest order terms. As a consequence our result covers short time existence of the Willmore flow of networks which corresponds to $\mu=0$.

\medskip

In all cases presented in the literature on the elastic flow of closed and  open curves the evolution exists globally in time~\cite{lin2, dalpoz, dapozspe, dziukkuwertsch, lin1, novok, poldenthesis}.
In the case of networks instead
recent results on the related minimization problem of the energy $E_\mu$
suggest different possible behaviours in the long term of our flow. 
There exist minimizing sequences of Theta--networks that converge to a point. Thus, in order to have a well--posed problem in the class of Theta--networks, one needs a further constraint.
One good option among several others is to
prescribe the angles at the triple junctions with at least one positive angle.
Recently it has been shown that there exists a minimizer in the class of Theta--networks in the case that all angles are prescribed to be of $120$ degrees~\cite{danovplu}.
To prove this one has to enlarge the class of Theta--networks fulfilling the angle condition to degenerate Theta--networks. A degenerate Theta--network consists of only two curves meeting in one quadruple point forming angles in pairs of $60$ and $120$ degrees. It is shown in~\cite{danovplu} that the global minimizer is not degenerate. But it is not excluded that a degenerate network could be a stationary point.
 In contrast to the static problem both the $C^0$- and the $C^1$--flow are interesting in the evolution. In particular, the ill--posedness of the minimization problem in the $C^0$--case suggests the possible onset of singularities in the long time behaviour. A dramatic change of topology is possible as one or more curves of the network may collapse to a point.

\medskip

In \cite{bargarnu} weak formulations for the Willmore flow of networks
have been proposed to study finite element
methods for the flow. 
Applications
to evolving non--linear splines are discussed 
and several numerical studies
demonstrate interesting features of the flow.

\medskip

In the following section we formulate the geometric evolution problem
for
elastic networks in a precise way. In Section 3 we show short time existence of a
unique smooth solution in the $C^0$--case using a maximal regularity result for linear parabolic
systems due to Solonnikov and a contraction argument. A main ingredient
in the proof is to show the Lopatinskii--Shapiro condition 
facing the difficulty that the boundary conditions mix the unknowns in a non-trivial way.
We also show geometric uniqueness in the $C^0$--case. In the final Section 4 we adapt the result to
the $C^1$--flow.

\section{The elastic flow for networks}

\subsection{Notation and definitions}

\subsubsection{Networks}
\begin{dfnz}
A network $\mathcal{N}$ is a connected set in the plane consisting of a finite 
union of sufficiently smooth curves $\mathcal{N}^i$.
\end{dfnz}

\medskip

Each curve $\mathcal{N}^i$ of the network
admits a parametrisation
$\gamma^i:[0,1]\to\mathbb{R}^2$.

\begin{itemize}
\item Let $k\in\mathbb{N}$ and $\alpha\in(0,1)$. A curve $\mathcal{N}^i$ is of class $C^{k}$ or $C^{k+\alpha}$ if it admits a parametrisation
$\gamma^i$ such that $\gamma^i$ is of class $C^k$ or $C^{k+\alpha}$, respectively. 
It 
is said to be regular if
$\gamma^i_x(x)\neq 0$ for every $x\in[0,1]$.
\item For a curve of class $C^1$
we define its unit tangent vector
$\tau^i:=\gamma^i_x / \vert \gamma^i_x\vert$ and its unit normal vector $\nu^i$
as the anticlockwise rotation by $\pi/2$ of its unit tangent vector.
\end{itemize}

We  write 
$\mathcal{N}=\bigcup_{i=1}^n\mathcal{N}^i$ when we consider a network 
as a geometric object. 
Instead  when we consider a parametrisation
of $\mathcal{N}$, we write 
$\gamma=(\gamma^1,\ldots,\gamma^n)$.

\medskip

In this paper we consider only $3$--networks, namely networks 
composed by three curves. In particular we define Triods and Theta--networks
as follows:
\begin{dfnz}
Consider an open   set $\Omega\subset\mathbb{R}^2$.
A Triod $\mathbb{T}$ is a network in $\Omega$ composed by three regular
at least $C^1$--curves with one endpoint belonging to the boundary of  $\Omega$
and the other three endpoints meeting 
in a triple junction $O$.
\end{dfnz}

We call $P^1, P^2, P^3$ the endpoints at $\partial\Omega$ and 
without loss of generality we consider a parame- trisation 
$\gamma=(\gamma^1,\gamma^2,\gamma^3)$
of the Triod such that
$O=\gamma^{1}\left(0\right)=\gamma^{2}\left(0\right)=\gamma^{3}\left(0\right)$
and  $P^{i}=\gamma^{i}\left(1\right)$ .

\begin{dfnz}
A Theta--network $\Theta$ is a network in $\mathbb{R}^2$ composed by three regular
at least $C^1$--curves that intersect each other  at their endpoints 
in two triple junctions. 
\end{dfnz}

Also in this case we may choose a 
parametrisation 
$\gamma=(\gamma^1,\gamma^2,\gamma^3)$ of a Theta--network in such a way that
$O^1=\gamma^{1}\left(0\right)=\gamma^{2}\left(0\right)=\gamma^{3}\left(0\right)$ 
and $O^2=\gamma^{1}\left(1\right)=\gamma^{2}\left(1\right)=\gamma^{3}\left(1\right)$.

We have defined two types of networks whose endpoints meet in junctions and we now give
a name to this property.

\begin{dfnz}[Concurrency condition]
Given either a Triod or a Theta--network
we say that the 
\emph{concurrency condition} is satisfied 
in a point $O$ of the network if the three curves meet at $O$
in a triple junction.
\end{dfnz}

\begin{dfnz}[Angle condition]
Given either a Triod or a Theta--network,
we say that the curves satisfy 
the \emph{angle condition}
at a junction if 
the concurrency condition is satisfied and
the angles between the unit tangent vectors to the 
curves are of $120$ degrees,
namely
$\sum_{i=1}^3\tau^i=0$. 
\end{dfnz}

To be consistent with~\cite{bargarnu} we call a $C^0$--network and a
$C^1$--network 
a network that satisfies the concurrency and the angle condition, respectively. In particular a Theta--network will also be called $C^0$--Theta--Network.

\subsubsection{Function spaces}\label{parabolicspaces}

In the whole paper we consider
parabolic H\"{o}lder spaces (see also~\cite[\S 11, \S 13]{solonnikov2})
which are defined as follows.ù

For a function $u:[0,T]\times [0,1]\to\mathbb{R}$ we define for $\rho\in (0,1)$ the semi--norms
$$
[ u]_{\rho,0}:=\sup_{(t,x), (\tau,x)}\frac{\vert u(t,x)-u(\tau,x)\vert}{\vert t-\tau\vert^\rho}\,,
$$
and
$$
[ u]_{0,\rho}:=\sup_{(t,x), (t,y)}\frac{\vert u(t,x)-u(t,y)\vert}{\vert x-y\vert^\rho}\,.
$$
Then we define for $k\in \{0,1,2,3,4\}$, $\alpha\in (0,1)$ the classical parabolic H\"older spaces
$$
C^{\frac{k+\alpha}{4}, k+\alpha}([0,T]\times[0,1])
$$
of all functions $u:[0,T]\times [0,1]\to\mathbb{R}$ that have continuous derivatives $\partial_t^i\partial_x^ju$ where $i,j\in\mathbb{N}$ are such that $4i+j\leq k$ for which the norm
\begin{equation}
\left\lVert u\right\rVert_{C^{\frac{k+\alpha}{4},k+\alpha}}:=\sum_{4i+j=0}^k\left\lVert\partial_t^i\partial_x^ju\right\rVert_\infty
+\sum_{4i+j=k}\left[\partial_t^i\partial_x^ju\right]_{0,\alpha}+\sum_{0<k+\alpha-4i-j<4}\left[\partial_t^i\partial_x^ju\right]_{\frac{k+\alpha-4i-j}{4},0}
\end{equation}
is finite.

We notice that the space $C^{\frac{\alpha}{4},\alpha}\left([0,T]\times[0,1]\right)$ is equal to the space
$$
C^{\frac{\alpha}{4}}\left([0,T];C^0([0,1])\right)\cap C^0\left([0,T];C^\alpha([0,1])\right)
$$
with equivalent norms.

The boundary data are in spaces of the form
$C^{\frac{k+\alpha}{4}, k+\alpha}([0,T]\times\{0,1\}, \mathbb{R}^m)$
which we identify with 
$C^{\frac{k+\alpha}{4}}([0,T], \mathbb{R}^{2m})$
via the isomorphism
$f\mapsto (f(t,0),f(t,1))^t$.

For $\beta>0$ and $T>0$ we define
$$
C_0^\beta([0,T];\mathbb{R}^n)=
\{f\in C^\beta([0,T];\mathbb{R}^n)\;\text{such that}\;f(0)=0\}\,.
$$

\begin{lemma}\label{notprovedbis}
For $\beta>\alpha>0$ and $T\in[0,1]$
we have the embedding 
$$C_0^\beta([0,T];\mathbb{R}^n)\hookrightarrow C_0^\alpha([0,T];\mathbb{R}^n)\,.$$
Moreover for all $f\in C_0^\beta([0,T];\mathbb{R}^n)$ it holds that
$$
\Vert f\Vert_{C^{\alpha}}\leq 2T^{\beta-\alpha} \Vert f\Vert_{C^{\beta}}\,.
$$
\end{lemma}
\begin{proof}
This follows directly from the fact that $f(0)=0$ and the definition of the H\"{o}lder norm.
\end{proof}

\subsection{Definition of the flow}

We consider the elastic energy functional $E_\mu$
for a $3$--network $\mathcal{N}$ defined as

\begin{equation}\label{eef}
E_\mu\left(\mathcal{N}\right):=\int_{\mathcal{N}}\left(k^{2}+\mu\right)ds
=\sum_{i=1}^3\int_{\mathcal{N}^{i}}\left((k^i)^{2}+\mu\right)ds^i\,,
\end{equation}
with $\mu\in (0,\infty)$, where $k^i$ is the curvature
of the curve $\mathcal{N}^i$
and $s^i$ its arclength parameter.
Here we use a sign convention for the curvature such that  $\tau_s=k\nu$.

\medskip

Our aim is to study the $L^2$--gradient flow of $E_\mu$
for a Triod or a Theta--network
in the case we require at the triple junctions the
concurrency or the angle conditions.

\begin{prob}[$C^0$--Flow]\label{geometricproblem}
Either a Triod or a Theta--network evolves by the geometric $L^2$--gradient flow 
of the elastic energy functional~\eqref{eef} requiring that 
the concurrency condition is satisfied at the triple junction(s).
\end{prob}
\begin{prob}[$C^1$--Flow]\label{generalgeometricproblemII}
A Triod or a Theta--network evolves by the geometric $L^2$--gradient flow 
of the elastic energy functional~\eqref{eef} requiring that 
 the angle condition is satisfied at the triple junction(s).
\end{prob}

In both problems the triple junctions are allowed to move.
We are looking for classical solutions to both Problem~\ref{geometricproblem} 
and Problem~\ref{generalgeometricproblemII}.

\medskip
We compute
the first variation of the functional $E_\mu$ for a  Triod or a Theta--network 
$\mathcal{N}$ admitting a regular $C^{4+\alpha}$--parametrisation $\gamma=(\gamma^1,\gamma^2,\gamma^3)$, compare \cite{bargarnu,   dallapluda} for details.

We get
\begin{align}\label{variation}
 \frac{d}{dt}E_\mu(\widetilde{\mathcal{N}}_t)_{|t=0}=&
 \sum_{i=1}^{3}\int_{\gamma^i}\left(2k_{ss}^{i}+\left(k^{i}\right)^{3}
 -\mu k^{i}\right)\left\langle \psi^{i},\nu^{i}\right\rangle\,ds\\
  & +\left.\sum_{i=1}^{3}2\left\langle \psi_{s}^{i},k^{i}\nu^i\right\rangle \right|_{0}^{1}
 +\left.\sum_{i=1}^{3}\left\langle \psi^{i},-2k_{s}^{i}\nu^i-\left( k^{i}\right)^{2}\tau^{i}
 +\mu\tau^{i}\right\rangle \right|_{0}^{1}\,,\nonumber
\end{align}
where $\widetilde{\mathcal{N}}_t$ is a variation of the network  $\mathcal{N}$
parametrised by $\tilde{\gamma}=(\tilde{\gamma}^1,\tilde{\gamma}^2,\tilde{\gamma}^3)$
with $\tilde{\gamma}^i=\gamma^i+t\psi^i$, $t\in\mathbb{R}$, $\psi^i\in 
C^\infty([0,1],\mathbb{R}^2)$ and $\psi^1(0)=\psi^2(0)=\psi^3(0)$.
In the case of a Theta--network we also require
$\psi^1(1)=\psi^2(1)=\psi^3(1)$ whereas in the case of a Triod we
prescribe  $\psi^i(1)=0$, $i=1,2,3$, in the Navier case and 
$\psi^i(1)=0$, $\psi^i_x(1)=0$, $i=1,2,3$ in the clamped case.
In the $C^1$--case there is a further condition on $\psi$ at the triple junction(s), see~\cite{bargarnu,   dallapluda} for details.

The previous computation tells us that the steepest descent with respect to the 
$L^2$--inner product  is given by a
normal velocity for each curve $\gamma^i$ equal to
$$
-\left(2k_{ss}^{i}+\left(k^{i}\right)^{3}-\mu k^{i}\right)\nu^{i}=:-A^i\nu^i\,.
$$

\medskip

Computing this quantity in terms of the parametrisation $\gamma^i$ we get
\begin{align}\label{Apart}
&A^i\nu^i=2k_{ss}^{i}\nu^{i}+\left(k^{i}\right)^{3}\nu^{i}-\mu k^{i}\nu^{i}\nonumber \\
&=2\frac{\gamma_{xxxx}^{i}}{\left|\gamma_{x}^{i}\right|^{4}}
-12\frac{\gamma^i_{xxx}\left\langle \gamma^i_{xx},\gamma^i_{x}\right\rangle }{\left|\gamma^i_{x}\right|^{6}}
-5\frac{\gamma^i_{xx}\left|\gamma^i_{xx}\right|^{2}}{\left|\gamma^i_{x}\right|^{6}}
-8\frac{\gamma_{xx}^{i}\left\langle \gamma_{xxx}^{i},\gamma_{x}^{i}\right\rangle }
{\left|\gamma_{x}^{i}\right|^{6}}
+35\frac{\gamma_{xx}^{i}\left\langle \gamma_{xx}^{i},\gamma_{x}^{i}\right\rangle ^{2}}
{\left|\gamma_{x}^{i}\right|^{8}}\nonumber\\
&+\left\langle -2\frac{\gamma_{xxxx}^{i}}{\left|\gamma_{x}^{i}\right|^{4}}+12\frac{\gamma_{xxx}^{i}\left\langle 
\gamma_{xx}^{i},\gamma_{x}^{i}\right\rangle }{\left|\gamma_{x}^{i}\right|^{6}}+5\frac{\gamma_{xx}^{i}\left|
\gamma_{xx}^{i}\right|^{2}}{\left|\gamma_{x}^{i}\right|^{6}}+8\frac{\gamma_{xx}^{i}\left\langle 
\gamma_{xxx}^{i},\gamma_{x}^{i}\right\rangle }{\left|\gamma_{x}^{i}\right|
^{6}}-35\frac{\gamma_{xx}^{i}\left\langle \gamma_{xx}^{i},\gamma_{x}^{i}\right\rangle ^{2}}{\left|\gamma_{x}^{i}\right|^{8}},\tau^{i}\right\rangle 
\tau^{i}\nonumber\\
&-\mu\frac{\gamma_{xx}^{i}}{\left|\gamma_{x}^{i}\right|^{2}}
+\left\langle \mu\frac{\gamma_{xx}^{i}}{\left|\gamma_{x}^{i}\right|^{2}},
\tau^{i}\right\rangle \tau^{i}\,.
\end{align}

\section{$C^0$--Flow}

Now we consider the $C^0$--Flow for a Theta--network.

From now on we
consider an initial $C^0$--Theta network $\Theta_0=\bigcup_{i=1}^3\sigma^i$ 
and we denote with $\tau^i_0,\nu^i_0, k_0$
its unit tangent, unit normal vectors and
its  curvature, respectively. 
Moreover we call  $\alpha_3,\alpha_1$ and $\alpha_2$ the angle between $\tau^1_0$ and  
 $\tau^2_0$, $\tau^2_0$ and  $\tau^3_0$, and $\tau^3_0$ and $\tau^1_0$, respectively.
Then basic trigonometric relations imply that the angles and tangents satisfy
\begin{align*}
\sin(\alpha_1)\tau^1_0+\sin(\alpha_2)\tau^2_0+\sin(\alpha_3)\tau^3_0=0\,,\\
\sin(\alpha_1)\nu^1_0+\sin(\alpha_2)\nu^2_0+\sin(\alpha_3)\nu^3_0=0\,.
\end{align*}

\subsection{Geometric problem}

To obtain an $L^2$--gradient flow of the functional $E_\mu$ for
a $C^0$--Theta--network
the following two conditions
need to be fulfilled at the boundary
(which are derived  from~\eqref{variation}, 
for a detailed proof see~\cite{bargarnu}) : 
\begin{equation}\label{boundaryvariation}
\left.\sum_{i=1}^{3}2\left\langle \psi_{s}^{i},k^{i}\nu^i\right\rangle \right|_{0}^{1}=0\quad\text{and}\quad
\left.\left\langle \psi^1,\sum_{i=1}^{3}-2k_{s}^{i}\nu^i-\left( k^{i}\right)^{2}\tau^{i}
 +\mu\tau^{i}\right\rangle \right|_{0}^{1}=0\,
\end{equation}
for all $\psi^i\in C^\infty([0,1],\mathbb{R}^2)$ with 
$\psi^1(y)=\psi^2(y)=\psi^3(y)$ for $y\in\{0,1\}$.

These conditions lead to natural boundary conditions which we now specify.

\begin{dfnz}[Natural boundary conditions]
We say that a Theta--network satisfies 
\begin{itemize}
\item  the
\emph{curvature condition} if
$k^i=0$ holds at the triple junctions;
\item the
\emph{third order condition} if $\sum_{i=1}^{3}\left(2k_{s}^{i}\nu^{i}-\mu\tau^{i}\right)=0$ 
holds at the triple junctions. 
\end{itemize}
\end{dfnz}

We notice that~\eqref{boundaryvariation}
is satisfied if and only if 
the curvature and the third order conditions are satisfied.
We now specify geometrically admissible initial networks which are needed
to solve the evolution problem in the class of smooth parametrisations.
In particular we need to specify compatibility conditions between
initial and boundary data.


\begin{dfnz}[Geometrically admissible initial network]\label{admg}
A Theta--network $\Theta_0$
is a 
geometrically admissible initial network
for Problem~\ref{geometricproblem}
 if
\begin{itemize}
\item at each triple junction there are at least two curves that form a strictly positive angle;
\item at each triple junction each curve has zero curvature (curvature condition);
\item at each triple junction it holds 
$\sum_{i=1}^{3}\left(2k_{0,s}^{i}\nu_0^{i}-\mu\tau_0^{i}\right)=0$
(third order condition);
\item  at each triple junction it holds that 
$$
\sin(\alpha_1) A_0^1+\sin(\alpha_2) A_0^2+\sin(\alpha_3) A_0^3=0\,,
$$
where $A^i_0=2k_{0,ss}^{i}+\left(k_0^{i}\right)^{3}-\mu k_0^{i}$.
\item there exists a parametrisation $\gamma=(\gamma^1,\gamma^2,\gamma^3)$
of  $\Theta_0$ such that every curve
$\gamma^i:[0,1]\to\mathbb{R}^2$ is of class $C^{4+\alpha}$
and regular.
\end{itemize}
\end{dfnz}

\begin{dfnz}[Solution of the geometric problem]\label{geometricsolution}
Let $\Theta_0$ be a geometrically admissible  initial network.
A time dependent family of Theta--networks $\Theta(t)$
 is a solution to
the $C^0$--flow with initial data $\Theta_0$ 
in a time interval $[0,T]$ with $T>0$
if and only if there exist
parametrisations $\gamma(t)$ of $\Theta(t)$, $t\in[0,T]$, 
$\gamma^i\in C^{\frac{4+\alpha}{4}, 4+\alpha}([0,T]\times[0,1],\mathbb{R}^2)\,,$
such that 
for every  $t\in\left[0,T\right),x\in\left[0,1\right]$ and for $i\in\{1,2,3\}$
the system
\begin{equation}\label{Theta0}
\begin{cases}
\begin{array}{lll}
\left\langle \gamma_t^{i}(t,x), \nu^i(t,x)\right\rangle
\nu^i(t,x) =-A^{i}(t,x)\nu^{i}(t,x)& &\text{motion,}\\
\gamma^{1}\left(t,y\right)=\gamma^{2}\left(t,y\right)=\gamma^{3}\left(t,y\right)&\text{for}\,y\in\{0,1\} 
&\text{concurrency condition,}\\
k^{i}(t,y)=0 &\text{for}\,y\in\{0,1\} &\text{curvature condition,}\\
\sum_{i=1}^{3}\left(2k_{s}^{i}\nu^{i}-\mu\tau^{i}\right)(t,y)=0 &\text{for}\,y\in\{0,1\}  &
\text{third order condition,}\\
\Theta(0)=\Theta_0& &
\text{initial data}
\end{array}
\end{cases}
\end{equation}
is satisfied.
\end{dfnz}

To specify the 
geometric evolution of the network it is enough to write the normal velocity of 
the curves. 
We remind that in our evolution problem the triple junctions are allowed to move which forces the curves to move also in tangential direction.
\begin{prop}
	The system~\eqref{Theta0}
	is a $L^2$-gradient flow of $E_\mu$ and in particular
	$E_\mu$ decreases in time.
\end{prop}
\begin{proof}
	This follows from~\cite[Theorem~2.1]{bargarnu}.
\end{proof}
 
\subsection{Analytic problem}

In this subsection we consider a time dependent family of 
Theta--networks parametrised by regular curves
$\gamma=(\gamma^1,\gamma^2,\gamma^3)
\in C^{\frac{4+\alpha}{4},4+\alpha}([0,T]\times[0,1])$
with $T>0$ and $\alpha\in(0,1)$.

\begin{dfnz}[Tangential velocity]
We define 
\begin{align}\label{Tang}
T^{i}&:=\left\langle 2\frac{\gamma^i_{xxxx}}{\left|\gamma^i_{x}\right|^{4}}
-12\frac{\gamma^i_{xxx}\left\langle \gamma^i_{xx},\gamma^i_{x}\right\rangle }
{\left|\gamma^i_{x}\right|^{6}}
-5\frac{\gamma^i_{xx}\left|\gamma^i_{xx}\right|^{2}}{\left|\gamma^i_{x}\right|^{6}}
\right. \nonumber \\
&\left. -8\frac{\gamma^i_{xx}\left\langle \gamma^i_{xxx},\gamma^i_{x}\right\rangle }
{\left|\gamma^i_{x}\right|^{6}}
+35\frac{\gamma^i_{xx}\left\langle \gamma^i_{xx},\gamma^i_{x}\right\rangle ^{2}}
{\left|\gamma^i_{x}\right|^{8}}
-\mu\frac{\gamma_{xx}^{i}}{\left|\gamma_{x}^{i}\right|^{2}},\tau^{i}\right\rangle \,.
\end{align}
\end{dfnz}

\begin{dfnz}[Second order condition]
We say that the parametrisation $\gamma$ of a Theta--network satisfies 
the \emph{second order condition}
if each curve satisfies $\gamma^i_{xx}=0$ at the triple junctions.
\end{dfnz}

\begin{rem}\label{secondorder}
We want to write
 the geometric problem 
as a system of PDEs.
This is obtained describing the curves
with the help of a parametrisation $\gamma$.
To have a well posed problem
 we need to fix a tangential velocity and to impose another condition
on the parametrisation of each curve at the boundary.
This is due to the fact that we need to specify some
tangential degrees of freedom.
These conditions have to be chosen in such a way that 
the geometric problem does not change 
(see below the discussion about geometric existence and uniqueness
in Subsection~\ref{geometricstuff}). 
As it turns out the boundary condition
\begin{equation}\label{tangentialsecondorder}
\left\langle \gamma_{xx},\tau \right\rangle=0
\end{equation}
is convenient in our case  (see Lemma~\ref{LopatinskiiShapiro}).
Together with the curvature condition the condition~\eqref{tangentialsecondorder} 
is equivalent to the second
order condition.
\end{rem}

Hence we consider the following problem:

\begin{prob}\label{latexproblem}
For every  $t\in\left[0,T\right),x\in\left[0,1\right]$ and for $i\in\{1,2,3\}$
\begin{equation}\label{thetaC0}
\begin{cases}
\begin{array}{lll}
\gamma_t^{i}(t,x)=-A^{i}(t,x)\nu^{i}(t,x)-T^{i}(t,x)\tau^{i}(t,x)& &\text{motion,}\\
\gamma^{1}\left(t,y\right)=\gamma^{2}\left(t,y\right)=\gamma^{3}\left(t,y\right)&\text{for}\,y\in\{0,1\} 
&\text{concurrency condition,}\\
\gamma_{xx}^{i}(t,y)=0 &\text{for}\,y\in\{0,1\} &\text{second order condition,}\\
\sum_{i=1}^{3}\left(2k_{s}^{i}\nu^{i}-\mu\tau^{i}\right)(t,y)=0 &\text{for}\,y\in\{0,1\}  &
\text{third order condition,}\\
\gamma^i(0,x)=\varphi^i(x)&  &
\text{initial data}
\end{array}
\end{cases}
\end{equation}
with $\varphi^i$ admissible initial parametrisation as defined in Definition~\ref{adm}.
\end{prob}

\begin{dfnz}[Compatibility conditions]
A parametrisation $\varphi$
of 
an initial network $\Theta_0$
satisfies
the compatibility conditions for system~\eqref{thetaC0}  if 
it satisfies the concurrency, second and third order conditions and
\begin{equation*}
A_0^{i}(y)\nu_0^{i}(y)+T_0^{i}(y)\tau_0^{i}(y)
=A_0^{j}(y)\nu_0^{j}(y)+T_0^{j}(y)\tau_0^{j}(y)
\end{equation*}
for $i,j\in\{1,2,3\}$ and $y\in\{0,1\}$
where $A^i_0$ and $T^i_0$ denote the equations~\eqref{Apart} and~\eqref{Tang}
applied to $\varphi^i$.
\end{dfnz}

\begin{dfnz}[Admissible initial parametrisation]\label{adm}
A parametrisation $\varphi$
of an initial Theta--network $\Theta_0$
is an admissible initial parametrisation for system~\eqref{thetaC0}
if 
\begin{itemize}
\item for $i\in\{1,2,3\}$ and for some $\alpha\in(0,1)$ we have $\varphi^i\in C^{4+\alpha}([0,1])$;
\item each curve $\varphi^i$ is regular;
\item for $y\in\{0,1\}$ we have that 
$\mathrm{span}\{\nu^1_0(y),\nu^2_0(y),\nu^3_0(y)\}=\mathbb{R}^2$;
\item the compatibility conditions for system~\eqref{thetaC0} are satisfied. 
\end{itemize}
\end{dfnz}

\subsection{Existence of the linearised system}

From now on we fix an initial  Theta--network $\Theta_0$
parametrised by an admissible initial parametrisation $\varphi$.

\subsubsection{Linearisation}\label{linearisation}

Consider the system defined in~\eqref{thetaC0}. 
Omitting the dependence on $(t,x)$ 
we write our motion equation in the form
\begin{equation}
\gamma^i_t=-2\frac{\gamma^i_{xxxx}}{\vert \gamma^i_x\vert^4}+\tilde{f}(\gamma^i_{xxx},\gamma^i_{xx},\gamma^i_x)
\end{equation}
for $i=1,2,3$.
We linearise the highest order terms of the previous equations around the initial parametrisation obtaining
\begin{align}
\gamma^i_t
+\frac{2}{\vert\varphi^i_x\vert^4}\gamma^i_{xxxx}
&=\left(\frac{2}{\vert\varphi^i_x\vert^4} -\frac{2}{\vert\gamma^i_x\vert^4}\right)\gamma^i_{xxxx} 
+\tilde{f}(\gamma^i_{xxx},\gamma^i_{xx},\gamma^i_x)=:f^i(\gamma^i_{xxxx},\gamma^i_{xxx},
\gamma^i_{xx},\gamma^i_{x})\,.
\end{align}
Omitting the dependence on $(t,y)$ with $y\in\{0,1\}$
we linearise the highest order terms of the boundary conditions.
The concurrency condition and the second order condition are already linear:
\begin{equation*}
\gamma^{1}=\gamma^{2}=\gamma^{3} \quad\text{and}\quad
\gamma^i_{xx}=0  \quad\text{for}\;i\in\{1,2,3\}\,.
\end{equation*}
The third order condition is of the form
\begin{equation}\label{thirdorder}
\sum_{i=1}^3\frac{1}{\vert \gamma^i_x\vert^3}
\left\langle \gamma^i_{xxx},\nu^i\right\rangle \nu^i 
+h^i(\gamma^i_x)=0\,.
\end{equation}
The linearised version of the highest order term in the third order condition~\eqref{thirdorder}
is the following:
\begin{align}
-\sum_{i=1}^3\frac{1}{\vert \varphi^i_x\vert^3}
\left\langle \gamma^i_{xxx},\nu_0^i\right\rangle \nu_0^i =
-\sum_{i=1}^3\frac{1}{\vert \varphi^i_x\vert^3}
\left\langle \gamma^i_{xxx},\nu_0^i\right\rangle \nu_0^i
+\sum_{i=1}^3\frac{1}{\vert \gamma^i_x\vert^3}
\left\langle \gamma^i_{xxx},\nu^i\right\rangle \nu^i 
+h^i(\gamma^i_x)=:b(\gamma)
\,.
\end{align}

The linearised system associated to~\ref{thetaC0} is given by
\begin{equation}\label{lyntheta}
\begin{cases}
\begin{array}{lll}
\gamma^i_t+\frac{2}{\vert\varphi^i_x\vert^4}\gamma^i_{xxxx}&=f^i
&\;\text{motion,}\\
\gamma^{1}-\gamma^{2}&=0 &\;\text{concurrency,}\\
\gamma^{1}-\gamma^{3}&=0 &\;\text{concurrency,}\\
\gamma^i_{xx}&=0 &\;\text{second order,}\\
-\sum_{i=1}^3\frac{1}{\vert \varphi^i_x\vert^3}
\left\langle \gamma^i_{xxx},\nu_0^i\right\rangle \nu_0^i&=b
&\;\text{third order,}\\
\gamma^{i}(0)&=\psi^{i} &\;\text{initial condition}\\
\end{array}
\end{cases}
\end{equation}
for $i\in\{1,2,3\}$
where we have omitted the dependence on $(t,x)\in[0,T]\times[0,1]$ in the  motion
equation, on $(t,y)\in[0,T]\times\{0,1\}$ for the boundary conditions 
and on $x\in[0,1]$ for the initial condition, respectively. Here $(f,b,\psi)$ is a general right hand side.
\begin{dfnz}\label{lincompcond}[Linear compatibility conditions]
Let $(f,b)$ be a given right hand side to the linear system~\eqref{lyntheta}. A function $\psi\in C^{4+\alpha}\left([0,1];\left(\mathbb{R}^2\right)^3\right)$ satisfies the linear compatibility conditions for the linear system~\eqref{lyntheta} with respect to $(f,b)$ if for $y\in\{0,1\}$ and $i,j\in\{1,2,3\}$ the equalities
\begin{align*}
\psi^i(y)&=\psi^j(y) \,,\\
\psi^i_{xx}(y)&=0 \,,\\
-\sum_{i=1}^3\frac{1}{\vert \varphi^i_x\vert^3}\left\langle \psi^i_{xxx}(y),\nu_0^i \right\rangle 
\nu_0^i &=b(0,y) \\
\frac{2}{\vert \varphi^i_x\vert^4}\psi^i_{xxxx}(y)-f^i(0,y)&=
\frac{2}{\vert \varphi^j_x\vert^4}\psi^j_{xxxx}(y)
-f^j(0,y) 
\end{align*}
hold.
\end{dfnz}
%
\begin{rem}
In the end we will be interested in the case $\psi=\varphi$.
\end{rem}
We show existence of the linear parabolic boundary value problem~\eqref{lyntheta} using the 
theory of Solonnikov~\cite{solonnikov2}. 

\subsubsection{Parabolic system in the sense of Solonnikov}\label{Solo}

Following the notation in~\cite{solonnikov2}
we let $\gamma^i=(u^i,v^i)$ and $\varphi^i=(\tilde{\varphi}^i, \hat{\varphi}^i)$.
As we are working with a fourth order operator we have $b=2$ and $r=6$.
We write our motion equation in the form
\begin{equation}
\begin{cases}
\begin{array}{ll}
u^i_t+\frac{2}{\vert\varphi^i_x\vert^4}u^i_{xxxx}=a^i, &\;\text{}\\
v^i_t+\frac{2}{\vert\varphi^i_x\vert^4}v^i_{xxxx}=b^i.&\;\text{}
\end{array}
\end{cases}
\end{equation}
Notice that $\gamma(t,x)$  is the vector with components $(u^1,v^1,u^2,v^2,u^3,v^3)$
and $f(t,x)$ has  components $(a^1,b^1,a^2,b^2,a^3,b^3)$.
Introducing the matrix $\mathcal{L}(x,t,\partial_x,\partial_t)$ given by
\begin{equation}
\begin{bmatrix}
\partial_t+\frac{2}{\vert \varphi^1_x\vert^4}\partial^4_x & 0 & 0 & 0 & 0 & 0 \\
0 & \partial_t+\frac{2}{\vert \varphi^1_x\vert^4}\partial^4_x  & 0 & 0 & 0 & 0 \\
0 & 0 & \partial_t+\frac{2}{\vert \varphi^2_x\vert^4}\partial^4_x  & 0 & 0 & 0 \\
0 & 0 & 0 & \partial_t+\frac{2}{\vert \varphi^2_x\vert^4}\partial^4_x  & 0 & 0 \\
0 & 0 & 0 & 0 & \partial_t+\frac{2}{\vert \varphi^3_x\vert^4}\partial^4_x  & 0 \\
0 & 0 & 0 & 0 & 0 & \partial_t+\frac{2}{\vert \varphi^3_x\vert^4}\partial^4_x   \\
\end{bmatrix}
\end{equation}
the previous system of
six equations in the unknown functions $u^1,v^1,u^2,v^2,u^3,v^3$
can be written in the form
\begin{equation}
\mathcal{L}\gamma=f\,.
\end{equation}

With the choice $s_k=4$ for every $k\in\{1,\ldots, 6\}$ and $t_j=0$ for every $j\in\{1,\ldots, 6\}$
the conditions of~\cite[page 8]{solonnikov2} are satisfied.
Following the classical theory in~\cite{solonnikov2} we call
$\mathcal{L}_0$ the principal part of the matrix 
$\mathcal{L}$.
With our choice of $s_k$ and $t_j$ we have $\mathcal{L}_0=\mathcal{L}$
and 
$$
\mathrm{det}\mathcal{L}_0(x,t,i\xi,p)=
\prod _{j=1}^3\left(  \frac{2}{\vert \varphi^j_x\vert}\xi^4+p\right) ^2\,.
$$
This is a polynomial of degree six in $p$ with roots of multiplicity two of the form 
$p_j=-\frac{2}{\vert \varphi^j_x\vert^4}\xi^4$ with $j=1,2,3$.
Choosing $\delta = \min\left\{\frac{2}{\vert \varphi^j_x\vert^4}: j=1,2,3\right\}$
the system is parabolic in the sense of 
Solonnikov (see~\cite[page 8]{solonnikov2}).

\subsubsection{Complementary conditions}
The complementary condition in~\cite[page 11]{solonnikov2} follows from
the Lopatinskii--Shapiro condition  (see~\cite[pages 11--15]{eidelman2}).
We prove that the Lopatinskii--Shapiro condition is satisfied in the
boundary  point $y=0$. The case $y=1$ can be treated analogously. 

\begin{dfnz}
Let $\lambda\in\mathbb{C}$ with $ \Re(\lambda)>0$ be arbitrary.
The Lopatinskii--Shapiro condition
demands that every solution $(\gamma^i)_{i=1,2,3}\in C^4([0,\infty),(\mathbb{C}^2)^3)$ to
the system of ODEs
\begin{equation}\label{LopatinskiiShapirosystem}
\begin{cases}
\begin{array}{llll}
\lambda \gamma^i(x)+\frac{1}{\vert\varphi^i_x
\vert^4}\gamma^i_{xxxx}(x)&=0&\;x\in[0,\infty),
i\in\{1,2,3\}&\;\text{motion,}\\
\gamma^{1}(0)-\gamma^{2}(0)&=0 &\; &\;\text{concurrency,}\\
\gamma^{2}(0)-\gamma^{3}(0)&=0 &\; &\;\text{concurrency,}\\
\gamma^i_{xx}(0)&=0 &\;i\in\{1,2,3\}&\;\text{second order,}\\
\sum_{i=1}^3\frac{1}{\vert \varphi_x^i\vert ^3}
\left\langle \gamma^i_{xxx}(0),\nu_0^i\right\rangle \nu_0^i&=0 &\; &\;\text{third order}\\
\end{array}
\end{cases}
\end{equation}
which satisfies $\lim_{x\to\infty}\lvert \gamma^i(x)\rvert=0$ has to be the trivial solution.
\end{dfnz}

In system~\eqref{LopatinskiiShapirosystem} and in the proof of the following lemma we use the notation
$\nu_0^i:=\nu^i_0(0)$ and $\varphi^{i}_x:=\varphi^{i}_x(0)$.

\begin{lemma}\label{LopatinskiiShapiro}
The Lopatinskii--Shapiro condition is satisfied.
\end{lemma}
\begin{proof}
Let $(\gamma^{i})_{i=1,2,3}$ be a solution to~\ref{LopatinskiiShapirosystem} 
such that $\lim_{x\to\infty}\lvert \gamma^i(x)\rvert=0$. 
With $\gamma^{i}$ also all derivatives decay to zero for $x$ tending to
infinity which follows from the specific exponential form of the
solutions to (\ref{LopatinskiiShapirosystem}).
We test the motion equation by 
$\vert\varphi^i_x\vert\left\langle \overline{\gamma}^i(x),\nu^i_0\right\rangle \nu^i_0$,
integrate twice by parts and sum to get 
\begin{align}\label{afterint}
0&=\sum_{i=1}^3\lambda\vert\varphi^i_x\vert \int_0^\infty\vert \left\langle\gamma^i(x), \nu^i_0\right\rangle\vert^2
\,\mathrm{d}x
+\sum_{i=1}^3\frac{1}{\vert\varphi^i_x\vert^3} 
\int_0^\infty\vert \left\langle \gamma_{xx}^i(x),\nu^i_0\right\rangle\vert^2\,\mathrm{d}x\\
&+\sum_{i=1}^3\frac{1}{\vert\varphi^i_x\vert^3}\left\langle\overline{\gamma}^i(0),\nu^i_0\right\rangle \left\langle \gamma^i_{xxx}(0),\nu^i_0 \right\rangle 
-\sum_{i=1}^3\frac{1}{\vert\varphi^i_x\vert^3}
\left\langle\overline{\gamma}_x^i(0),\nu^i_0\right\rangle \left\langle \gamma^i_{xx}(0),\nu^i_0 \right\rangle \,,\nonumber
\end{align}
where we have already used the decay at infinity.
Using the concurrency and the third order condition in the first boundary term we get
\begin{equation}
\sum_{i=1}^3\frac{1}{\vert\varphi^i_x\vert^3}\left\langle\overline{\gamma}^i(0),\nu^i_0\right\rangle \left\langle \gamma^i_{xxx}(0),\nu^i_0 \right\rangle =
\left\langle\overline{\gamma}^1(0),
\sum_{i=1}^3\frac{1}{\vert\varphi^i_x\vert^3}
\left\langle \gamma^i_{xxx}(0),\nu^i_0 \right\rangle\nu^i_0\right\rangle
 =0\,.
\end{equation}
The second boundary term vanishes due to the second order condition.
Now we take the real part of~\eqref{afterint} and we obtain
$\left\langle \gamma^i(x),\nu^i_0\right\rangle =0$ 
for $x\in[0,\infty)$ and $i\in\{1,2,3\}$.

In particular, using the concurrency condition, we find that
$\left\langle \gamma^1(x),\sum_{i=1}^3a^i\nu^i_0\right\rangle =0$  for all $a^i\in\mathbb{C}$.
As $\varphi$ is an admissible initial parametrisation, this enforces $\gamma^i(0)=0$.
 
In the same way as before, testing the motion equation by 
$\vert\varphi^i_x\vert\left\langle \overline{\gamma}^i(x),\tau^i_0\right\rangle \tau^i_0$ 
where $\tau^i_0$ is the unit tangent vector in $\varphi^i(0)$,
we get

\begin{align}\label{withtang}
0&=\sum_{i=1}^3\lambda\vert\varphi^i_x\vert \int_0^\infty\vert \left\langle\gamma^i(x), \tau^i_0\right\rangle\vert^2
\,\mathrm{d}x
+\sum_{i=1}^3\frac{1}{\vert\varphi^i_x\vert^3} 
\int_0^\infty\vert\left\langle \gamma_{xx}^i(x), \tau^i_0\right\rangle\vert^2\,\mathrm{d}x\\
&+\sum_{i=1}^3\frac{1}{\vert\varphi^i_x\vert^3}\left\langle\overline{\gamma}^i(0),\tau^i_0\right\rangle \left\langle \gamma^i_{xxx}(0),\tau^i_0 \right\rangle 
-\sum_{i=1}^3\frac{1}{\vert\varphi^i_x\vert^3}
\left\langle\overline{\gamma}_x^i(0),\tau^i_0\right\rangle \left\langle \gamma^i_{xx}(0),\tau^i_0 \right\rangle \,.\nonumber
\end{align}
Again the second boundary term vanishes due to the second order condition
and the first boundary term vanishes due to the fact that we have previously 
obtained $\gamma^i(0)=0$.
Considering again only the real part of~\eqref{withtang}
we find for all $x\in[0,\infty)$ and $i\in\{1,2,3\}$ that 
$\left\langle\gamma^i(x), \tau^i_0\right\rangle =0$. Hence we conclude that $\gamma^i(x)=0$ for every 
$x\in[0,\infty)$ and $i\in\{1,2,3\}$.
\end{proof}

Finally to check the complementary conditions stated in~\cite[page 12]{solonnikov2}
 for the initial data
we observe, using again the notation of \cite{solonnikov2},  that the $6\times 6$ matrix $[C_{\alpha j}]$ is the identity matrix.
With the choice $\gamma_{\alpha j}=0$ for every $\alpha\in\{1,\ldots,6\}$ and $j\in\{1,\ldots,6\}$
we obtain $\rho_\alpha=0$ and $C_0=Id$.
Moreover the rows of the matrix $\mathcal{D}(x,p)=\hat{\mathcal{L}_0}(x,0,0,p)=p^5 Id$
are linearly independent modulo the polynomial $p^6$.

\subsubsection{Existence theorem for the linearised system}

Adapting~\cite[Theorem 4.9]{solonnikov2} to our situation we obtain the following existence result
for the system~\eqref{lyntheta}:
\begin{teo}\label{linearexistence}
Let $\alpha\in(0,1)$.
For every $T>0$ the 
Problem~\eqref{lyntheta} has a unique solution $\gamma\in C ^{\frac{4+\alpha}{4},{^{4+\alpha}}}([0,T]\times[0,1];(\mathbb{R}^2)^3)$
provided that
\begin{itemize}
\item $f^i\in C ^{\frac\alpha4,{^{\alpha}}}([0,T]\times[0,1];\mathbb{R}^2)$ 
for $i\in\{1,2,3\}\,$;
\item $b\in C  ^{\frac{1+\alpha}{4}}([0,T];\mathbb{R}^4)\,$;
\item $\psi\in C^{4+\alpha}\left([0,1];\left(\mathbb{R}^2\right)^3\right)$ and
\item $\psi$ satisfies the linear compatibility conditions~\ref{lincompcond} with respect to $(f,b)$.
\end{itemize}
Moreover, for all $T>0$ there exists a $C(T)>0$ such that for all
solutions  the inequality
\begin{equation}
\sum_{i=1}^3 \Vert \gamma^i\Vert_{\frac{(4+\alpha)}{4},^{4+\alpha}} 
\leq C(T)\left( 
\sum_{i=1}^3 \Vert f^i\Vert_{\frac\alpha4,^{\alpha}}+
\Vert b\Vert_{\frac{1+\alpha}{4}}+
\Vert \psi\Vert_{4+\alpha}
\right)
\end{equation} 
holds.
\end{teo}

\subsection{Short time existence of the analytic problem}\label{contractionargu}

In the following we fix a coefficient $\alpha\in (0,1)$ and an admissible initial parametrisation $\varphi$ 
with $\Vert \varphi \Vert_{C^{4+\alpha}([0,1])}=R$.
\begin{dfnz}
For  $T>0$
we consider the  linear spaces
\begin{align*}
\mathbb{E}_T:=\{&\gamma\in 
C  ^{\frac{4+\alpha}{4},{^{4+\alpha}}}([0,T]\times[0,1];(\mathbb{R}^2)^3)
\;\text{such that for}\;i\in\{1,2,3\},   t\in[0,T]\,,\\
 &y\in\{0,1\}\;\text{it holds}\,
\gamma^1(t,y)=\gamma^2(t,y)=\gamma^3(t,y),\gamma^i_{xx}(t,y)=0\}
\,,\\
\mathbb{F}_T:=\{&(f,b,\psi)\in 
C  ^{\frac{\alpha}{4},{^{\alpha}}}([0,T]\times[0,1];(\mathbb{R}^2)^3)
\times 
C  ^{\frac{1+\alpha}{4}}([0,T];\mathbb{R}^4)
\times  C^{4+\alpha}\left([0,1];\left(\mathbb{R}^2\right)^3\right)\\
&\text{such that the linear compatibility conditions hold}
\}
\,,
\end{align*}
endowed with the induced norms.
\end{dfnz}

\begin{lemma}
The map $L_{T}:\mathbb{E}_T\to \mathbb{F}_T$ defined by
$$
L_{T}(\gamma)=
\begin{pmatrix}
\left(\gamma^i_t+\frac{2}{\vert\varphi^i_x\vert^4}\gamma^i_{xxxx}\right)_{i\in\{1,2,3\}}\\
-\mathrm{tr}_{\partial[0,1]}\sum_{i=1}^3\frac{1}{\vert \varphi^i_x\vert^3}
\left\langle \gamma^i_{xxx},\nu_0^i\right\rangle \nu_0^i\\
\gamma_{\vert t=0}
\end{pmatrix}
$$
is a continuous isomorphism.
\end{lemma}
\begin{proof}
The map $L_{T}$ is clearly well defined, bijective 
and continuous due to Theorem~\ref{linearexistence}.
\end{proof}

\begin{rem}
Notice that as  $\varphi$ is regular, there exists a constant $C>0$ such that
$$
\inf_{x\in[0,1]}\vert \varphi_x(x)\vert \geq C\,.
$$
As a consequence
we have 
$$
\sup_{x\in[0,1]}\frac{1}{\vert \varphi_x(x)\vert} \leq \frac{1}{C}\,,
$$
and 
for every $j\in\mathbb{N}$ there hold
$$
\left\lVert \frac{1}{\vert \varphi^i_x\vert^j}
\right\rVert_{C^{\alpha}([0,1])}
\leq \left(\frac{\Vert \varphi_x\Vert_{C^{\alpha}([0,1])}}{C^2}\right)^j
\leq \left(\frac{R}{C^2}\right)^j\,,
$$
and
$$
\left\lVert \frac{1}{\vert \varphi^i_x\vert^j}\right\rVert_{C^{1+\alpha}([0,1])}
\leq \widetilde{C}(R,C)\,.
$$
\end{rem}

\begin{dfnz}
We define the affine spaces
\begin{align*}
\mathbb{E}^\varphi_T=\{&\gamma\in 
\mathbb{E}_T\,\text{such that }\,\gamma_{\vert t=0}=\varphi\}
\,,\\
\mathbb{F}^\varphi_T=\{&(f,b)\,\text{such that }\, (f,b,\varphi)\in\mathbb{F}_{T}\}
\times\{\varphi\}
\,.
\end{align*}
\end{dfnz}

In the following 
we denote by $B_M$ the open ball of radius $M$ and center $0$ in $\mathbb{E}_T$.

\begin{lemma}
There exists $\tilde{t}(M, C)\in(0,1]$ such that 
for all
$\gamma\in \mathbb{E}^\varphi_T\cap B_M$
and for all $t\in[0,\tilde{t}(M,C)]$ we have
\begin{equation}\label{inverseest}
\inf_{x\in[0,1]}\vert \gamma_x(t,x)\vert\geq\frac{C}{2}\,.
\end{equation}
\end{lemma}
\begin{proof}
We have
$$
\vert \gamma_x(t,x)\vert\geq \vert\varphi_x(x)\vert -\vert \gamma_x(t,x)-\varphi_x(x)\vert
$$
with $\vert \varphi_x(x)\vert\geq C$
and $\vert \gamma_x(t,x)-\varphi_x(x)\vert\leq \left[\gamma_x\right] _{\beta,0}t^\beta
\leq Mt^\beta$ with $\beta=\frac{3}{4}+\frac{\alpha}{4}$.
Passing to the infimum one gets the claim.
\end{proof}

As a consequence every $\gamma \in \mathbb{E}^\varphi_T\cap B_M$
is a regular curve for $T\in [0,\tilde{t}(M,C)]$.
Moreover 
we find for every $\gamma\in \mathbb{E}^\varphi_T\cap B_M$ and 
$t\in[0,\tilde{t}(M,C)]$ that
\begin{equation}\label{lowbound}
\sup_{x\in[0,1]}\frac{1}{\vert \gamma_x(t,x)\vert}\leq \frac{2}{C}\,.
\end{equation}

\begin{lemma}\label{regularityest}
Let $M$ be positive and
$\tilde{t}(M,C)$  be as above.
Then for every $\gamma\in \mathbb{E}^\varphi_T\cap B_M$ with $T\in[0,\tilde{t}(M,C)]$
$$
\left\lVert \frac{1}{\vert \gamma^i_x\vert^j}\right\rVert_{C^{\frac{\alpha}{4},\alpha}([0,T]\times[0,1])}
\leq \left(\frac{4M}{C^2}\right)^j\,.
$$
\end{lemma}
\begin{proof}
For $j=1$ the result follows directly combining
 the estimate~\eqref{lowbound} with
 the definition of the norm 
$\Vert\cdot\Vert_{C^{\frac{\alpha}{4},\alpha}([0,T]\times[0,1])}$.
The case $j\geq 2$ follows from multiplicativity of the norm. 
\end{proof}

\begin{rem}
Let again $M$ be positive and
$\tilde{t}(M,C)$  be as above.

Then for every $T\in(0,\tilde{t}(M,C)]$,  $\gamma\in \mathbb{E}^\varphi_T\cap B_M$, $y\in\{0,1\}$ and $j\in\mathbb{N}$ it holds
\begin{equation}\label{usefulcone}
\left\lVert \frac{1}{\vert \gamma^i_x(y)\vert^j}\right\rVert_{C^{\frac{1+\alpha}{4}}([0,T])}
\leq \tilde{C}(R,C)\,.
\end{equation}
This can be shown as in Lemma~\ref{regularityest} using the definition of the norm of $C^{\frac{1+\alpha}{4}}([0,T])$, estimate~\eqref{lowbound} and the inequality $[f]_{\frac{1+\alpha}{4},0}\leq [f]_{\frac{3+\alpha}{4},0}$.
\end{rem}

\begin{lemma}\label{welldefN}
Let $T\in(0,\tilde{t}(M,C)]$. Then 
the maps 
\begin{align*}
N_{T,1}:&
\begin{cases}
\mathbb{E}^\varphi_T &\to C  ^{\frac\alpha4,
{^{\alpha}}}([0,T]\times[0,1];(\mathbb{R}^2)^3),\\
\gamma&\mapsto
f(\gamma):=(f^i(\gamma^i))_{i=1,2,3},
\end{cases}\\
N_{T,2}:&
\begin{cases}
\mathbb{E}^\varphi_T&\to  C^{\frac{1+\alpha}{4}}([0,T];\mathbb{R}^4), \\
\gamma&\mapsto b(\gamma)
\end{cases}
\end{align*}
are well defined where the functions $f^i(\gamma^i):=f^i(\gamma^i_{xxxx},\gamma^i_{xxx},\gamma^i_{xx},\gamma^i_x)$ and $b(\gamma):=b(\gamma_{xxx},\gamma_x)$ are defined in the linearisation~\ref{linearisation}.
\end{lemma}
\begin{proof}
As $f^i$ and $b$ are given by 
\begin{align*}
f^i(\gamma^i_{xxxx},\gamma^i_{xxx},\gamma^i_{xx},\gamma^i_x)=&\left(\frac{2}{\vert\varphi^i_x\vert^4} -\frac{2}{\vert\gamma^i_x\vert^4}\right)\gamma^i_{xxxx}+12\frac{\gamma_{xxx}\left\langle \gamma_{xx},\gamma_{x}\right\rangle }{\left|\gamma_{x}\right|^{6}}
+5\frac{\gamma_{xx}\left|\gamma_{xx}\right|^{2}}{\left|\gamma_{x}\right|^{6}}\\
&+8\frac{\gamma_{xx}^{i}\left\langle \gamma_{xxx}^{i},\gamma_{x}^{i}\right\rangle }
{\left|\gamma_{x}^{i}\right|^{6}}
-35\frac{\gamma_{xx}^{i}\left\langle \gamma_{xx}^{i},\gamma_{x}^{i}\right\rangle ^{2}}
{\left|\gamma_{x}^{i}\right|^{8}}+\mu\frac{\gamma_{xx}^{i}}{\vert\gamma_x^{i}\vert^2}\,,
\end{align*}
and
\begin{align*}
b(\gamma^i_{xxx},\gamma^i_x)=&-\sum_{i=1}^3\frac{1}{\vert \varphi^i_x\vert^3}
\left\langle \gamma^i_{xxx},\nu_0^i\right\rangle \nu_0^i
+\sum_{i=1}^3\frac{1}{\vert \gamma^i_x\vert^3}
\left\langle \gamma^i_{xxx},\nu^i\right\rangle \nu^i -\frac{\mu}{2}\sum_{i=1}^{3}\frac{\gamma^i_x}{\vert\gamma^i_x\vert}\,,
\end{align*}
the Banach--algebra property of the
H\"older spaces immediately implies that $N_{T,1}$ and $N_{T,2}$ are well--defined.
\end{proof}

\begin{lemma}
The map $N_T$ given by $\gamma\,\mapsto 
(N_{T,1},N_{T,2},\cdot\vert_{t=0})(\gamma)$
is a mapping from $\mathbb{E}^\varphi_T$ to $\mathbb{F}^\varphi_T$.
\end{lemma}
\begin{proof}
We have only to show that $\varphi$ satisfies the
linear compatibility conditions with respect to $(N_{T,1},N_{T,2})(\gamma)$.
As $\varphi$ satisfies the third order condition, we have for $y\in\{0,1\}$
\begin{align*}
b(\gamma^i_{xxx},\gamma^i_x)(0,y)&=
-\sum_{i=1}^3\frac{1}{\vert \varphi^i_x\vert^3}
\left\langle \varphi^i_{xxx},\nu_0^i\right\rangle \nu_0^i
+\sum_{i=1}^3\frac{1}{\vert \varphi^i_x\vert^3}
\left\langle \varphi^i_{xxx},\nu_0^i\right\rangle \nu_0^i 
-\frac{\mu}{2}\sum_{i=1}^{3}\frac{\varphi^i_x}{\vert\varphi^i_x\vert}\\
&=-\sum_{i=1}^3\frac{1}{\vert \varphi^i_x\vert^3}
\left\langle \varphi^i_{xxx},\nu_0^i\right\rangle \nu_0^i\,.
\end{align*}
Moreover, the admissible initial parametrisation $\varphi$ satisfies for $i,j\in\{1,2,3\}$ and $y\in\{0,1\}$
\begin{align*}
\frac{2}{\vert \varphi^i_x\vert^4}\varphi^i_{xxxx}(y)-f^i(\gamma^i)(0,y)&=
A_0^{i}(y)\nu_0^{i}(y)+T_0^{i}(y)\tau_0^{i}(y)=A_0^{j}(y)\nu_0^{j}(y)+T_0^{j}(y)\tau_0^{j}(y)\\
&=\frac{2}{\vert \varphi^j_x\vert^4}\varphi^j_{xxxx}(y)
-f^j(\gamma^j)(0,y)\,.
\end{align*}
\end{proof}

\subsubsection{Contraction estimates}

\begin{dfnz}
We define the mapping $K_T:\mathbb{E}^\varphi_T\to\mathbb{E}^\varphi_T$
as $K_T:=L_T^{-1}N_T$.
\end{dfnz}
Clearly the map $K_T$ is well defined
for all $T\in (0,\tilde{t}(M,C)]$
since $L_T^{-1}(\mathbb{F}_T^\varphi)\subset \mathbb{E}^\varphi_T$.

\begin{lemma}\label{notproved}
Let $k\in\{1,2,3\}$, $T\in [0,1]$ and $\gamma,\tilde{\gamma}\in\mathbb{E}_T^\varphi$. We denote by 
$\gamma^{(4-k)}, \tilde{\gamma}^{(4-k)}$ the 
$(4-k)$th  space derivative of $\gamma$ and $\tilde{\gamma}$, respectively.
Then there exist $\varepsilon_k>0$ and a constant $\tilde{C}$ independent of $T$ such that
$$
\left\lVert
\gamma^{(4-k)}-\tilde{\gamma}^{(4-k)}
\right\rVert_{C^{\frac{\alpha}{4},\alpha}}
\leq 
\tilde{C}T^{\varepsilon_k}\left\lVert
\gamma^{(4-k)}-\tilde{\gamma}^{(4-k)}
\right\rVert_{C^{\frac{k+\alpha}{4},k+\alpha}}
\leq \tilde{C}T^{\varepsilon_k}\left\lVert \gamma-\tilde{\gamma}
\right\rVert_{\mathbb{E}_T}\,.
$$
\end{lemma}
\begin{proof}
We let $k\in\{1,2,3\}$ be fixed and define 
$$
X_{T}^k:=\{g\in C^{\frac{k+\alpha}{4}, k+\alpha}\left([0,T]\times[0,1]\right):g(0,x)=0\text{ for all }x\in[0,1]\}.
$$
Note that for $\gamma,\tilde{\gamma}\in\mathbb{E}_T^\varphi$ we have that 
$\left(\gamma^{(4-k)}-\tilde{\gamma}^{(4-k)}\right)\in X_T^k$.
We fix some $\theta_k\in \left[\frac{\alpha}{k+\alpha},\frac{k}{k+\alpha}\right)$
such that $\theta_k(k+\alpha)$ is not an integer and let $\beta_k:=(1-\theta_k)\frac{k+\alpha}{4} > \frac{\alpha}{4}$.
By~\cite[Proposition 1.1.3]{lunardi1} the space $C^{\theta_k(k+\alpha)}([0,1])$ belongs to the class $J_{\theta_k}$ between $C([0,1])$ and $C^{k+\alpha}([0,1])$. Using \cite[Proposition 1.1.4]{lunardi1} this implies
$$
B\left([0,T];C^{k+\alpha}([0,1])\right)\cap C^{\frac{k+\alpha}{4}}\left([0,T];C([0,1])\right)\hookrightarrow C^{\beta_k}\left([0,T];C^{\theta_k(k+\alpha)}([0,1])\right)
$$
and hence also
$$
X_T^k\hookrightarrow C^{\beta_k}_0([0,T];C^{\theta_k(k+\alpha)}([0,1]))\hookrightarrow C^{\beta_k}_0([0,T];C^{\alpha}([0,1])) \hookrightarrow C^{\frac{\alpha}{4}}_0([0,T];C^{\alpha}([0,1])),
$$
where 
\[C^\delta_0\left([0,T];C^\mu([0,1])\right):=\{g\in C^\delta_0\left([0,T];C^\mu([0,1])\right):g(0,x)=0\text{ for all }x\in[0,1]\}\,.\]
Here all embedding constants are independent of $T\in[0,1]$. In particular we find for all $\gamma,\tilde{\gamma}\in \mathbb{E}_T^\varphi$ with $\varepsilon_k:=\beta_k-\frac{\alpha}{4}$ that
\begin{align*}
\Vert\gamma^{(4-k)}-\tilde{\gamma}^{(4-k)}\Vert_{C^{\frac{\alpha}{4},\alpha}}&
\leq \Vert\gamma^{(4-k)}-\tilde{\gamma}^{(4-k)}\Vert_{C^{\frac{\alpha}{4}}([0,T];C^\alpha[0,1])}\\
&\leq
T^{\varepsilon_k} 
\Vert\gamma^{(4-k)}
-\tilde{\gamma}^{(4-k)}
\Vert_{C^{\beta_k}([0,T];C^\alpha[0,1])}\\
&\leq T^{\varepsilon_k}\Vert\gamma^{(4-k)}-\tilde{\gamma}^{(4-k)}\Vert_{C^{\beta_k}([0,T];C^{\theta_k(k+\alpha)}([0,1])}\\
&\leq \tilde{C} T^{\varepsilon_k}\Vert\gamma^{(4-k)}-\tilde{\gamma}^{(4-k)}\Vert_{X_T^k}.
\end{align*}
\end{proof}


\begin{lemma}\label{uniformestimateestension}
For all $T_0>0$ there exists a constant $c(T_0)$ such that
$$
\sup_{T\in (0,T_0]}\Vert L_T^{-1}\Vert_{\mathcal{L}(\mathbb{F}_T,\mathbb{E}_T)}\leq c(T_0)\,.
$$
\end{lemma}
\begin{proof}
For $(f,b,\psi)\in\mathbb{F}_T$ we define the extension 
$E(f,b,\psi)$ as
$E(f,b,\psi):=(E_1f, E_2b,\psi)$ with
\begin{align*}
E_1f(t):=
\begin{cases}
f(t) &\text{for}\, t\in[0,T],\\
f\left(T\frac{T_0-t}{T_0-T}\right) &\text{for}\, t\in(T,T_0],\\
\end{cases}
\qquad
E_2b(t):=
\begin{cases}
b(t) &\text{for}\, t\in[0,T],\\
b\left(T\frac{T_0-t}{T_0-T}\right) &\text{for}\, t\in(T,T_0]\,.
\end{cases}
\end{align*}
It is clear that $E(f,b,\psi)\in\mathbb{F}_{T_0}$ and that 
$\Vert E \Vert_{\mathcal{L}(\mathbb{F}_T,\mathbb{F}_{T_0})}\leq 1$.

As $L^{-1}_{T_0}(E(f,b,\psi))_{\vert [0,T]}=L^{-1}_T(f,b,\psi)$
we find that
\begin{align*}
\Vert L^{-1}_T(f,b,\psi) \Vert_{\mathbb{E}_T}
&\leq \Vert L^{-1}_{T_0}(E(f,b,\psi)) \Vert_{\mathbb{E}_{T_0}}\\
&\leq \Vert L_{T_0}^{-1}\Vert_{\mathcal{L}(\mathbb{F}_{T_0},\mathbb{E}_{T_0})}
\Vert E(f,b,\psi)\Vert_{\mathbb{F}_{T_0}}\leq c(T_0)\Vert (f,b,\psi)\Vert_{\mathbb{F}_{T}}\,.
\end{align*}

\end{proof}
From now on we fix $T_0\in (0,1]$.

\begin{prop}\label{contraction}
For any positive radius $M$ there exists a time $T(M)\in (0,
\min\{
\tilde{t}(M,C)
,T_0\}
)$ such that for all $T\in (0,T(M)]$
the map 
$K_T:\mathbb{E}^\varphi_T\cap \overline{B_M}\to\mathbb{E}^\varphi_T$ 
is a contraction.
\end{prop}
\begin{proof}
We fix a positive radius $M$ and let $T$ be smaller than $
\min\{
\tilde{t}(M,C),T_0\}$.	
For $\gamma,\tilde{\gamma}\in\mathbb{E}_T^\varphi\cap \overline{B_M}$ we now estimate the
different components of the expression
$\Vert N_T(\gamma)-N_T(\tilde{\gamma})\Vert_{\mathbb{F}_T},$
namely
\begin{align*}
&\Vert N_{T,1}(\gamma)-N_{T,1}(\tilde{\gamma})\Vert_{C^{\frac{\alpha}{4}},\alpha} =
\Vert f(\gamma)-f(\tilde{\gamma})\Vert_{C^{\frac{\alpha}{4}},\alpha},\\
&\Vert N_{T,2}(\gamma)-N_{T,2}(\tilde{\gamma})\Vert_{C^{\frac{1+\alpha}{4}}}
=\Vert b(\gamma)-b(\tilde{\gamma})\Vert_{C^{\frac{1+\alpha}{4}}}.
\end{align*}

{\textbf{\textit{Step 1: Estimate on 
$\Vert f(\gamma)-f(\tilde{\gamma})\Vert_{C^{\frac{\alpha}{4}},\alpha}$}}}

The highest order term is given by
\begin{equation*}
\left(\frac{2}{\vert\varphi^i_x\vert^4} -\frac{2}{\vert\gamma^i_x\vert^4}\right)
\left(\gamma^i_{xxxx}-\tilde{\gamma}^i_{xxxx}\right)+
\left(\frac{2}{\vert\tilde{\gamma}^i_x\vert^4} -\frac{2}{\vert\gamma^i_x\vert^4}\right)
\tilde{\gamma}^i_{xxxx}.
\end{equation*}

Combining the identity 
$$
\frac{1}{\vert a \vert^4} -\frac{1}{\vert b \vert^4}=
\left(\vert b \vert-\vert a \vert \right)
\left(\frac{1}{\vert a \vert^2\vert b \vert}
+\frac{1}{\vert a \vert\vert b \vert^2}\right)
\left( \frac{1}{\vert a \vert^2}
+\frac{1}{\vert b \vert^2}\right)\,
$$
with Lemma~\ref{regularityest} and~\ref{notproved} we get 
\begin{align*}
\left\lVert
\left(\frac{2}{\vert\varphi^i_x\vert^4} -\frac{2}{\vert\gamma^i_x\vert^4}\right)\left(\gamma^i_{xxxx}
-\tilde{\gamma}^i_{xxxx}\right)
\right\rVert_{C^{\alpha,\frac{\alpha}{4}}}
\leq
\tilde{C}(M,C,R)T^{\varepsilon_3}
\Vert \gamma-\tilde{\gamma}\Vert_{\mathbb{E}_T}
\end{align*}
and 
\begin{align}\label{fractionestimate}
\left\lVert
\left(\frac{2}{\vert\tilde{\gamma}^i_x\vert^4} -\frac{2}{\vert\gamma^i_x\vert^4}\right)
\tilde{\gamma}^i_{xxxx}
\right\rVert_{C^{\alpha,\frac{\alpha}{4}}}
\leq
\tilde{C}(M,C)T^{\varepsilon_3}
\Vert \gamma-\tilde{\gamma}\Vert_{\mathbb{E}_T}\,.
\end{align}
All the other terms of  $ f(\gamma)-f(\tilde{\gamma})$ are of the  form
\begin{equation}\label{shape}
\frac{a\left\langle b,c\right\rangle }{\vert d\vert^j}
-\frac{\tilde{a}\langle \tilde{b},\tilde{c}\rangle }{\vert\tilde{d}\vert^{j}}
\end{equation}
with $j\in\{2,6,8\}$ 
and with $a,b,c,d,\tilde{a},\tilde{b},\tilde{c},\tilde{d}$ spacial derivatives up to order three
of $\gamma^i$ and $\tilde{\gamma}^i$, respectively.

Adding and subtracting the expression
$$
\frac{\tilde{a}\left\langle b,c\right\rangle }{\vert d\vert^j}+\frac{\tilde{a}\left\langle \tilde{b},c\right\rangle }{\vert d\vert^j}+
\frac{\tilde{a}\left\langle \tilde{b},\tilde{c}\right\rangle }{\vert d\vert^j}
$$
to~\eqref{shape},
we get 
\begin{equation}\label{lowerorder}
\frac{(a-\tilde{a})\left\langle b,c\right\rangle }{\vert d\vert^j}+
\frac{\tilde{a}\left\langle (b-\tilde{b}),c\right\rangle }{\vert d\vert^j}
+\frac{\tilde{a}\left\langle \tilde{b},(c-\tilde{c})\right\rangle }{\vert d\vert^j}
+\left(\frac{1}{\vert d\vert^j}-\frac{1}{\vert \tilde{d}\vert^j}\right)
\tilde{a}\left\langle \tilde{b},\tilde{c}\right\rangle\,.
\end{equation}
With the help of Lemma~\ref{notproved} we can estimate the first term of~\eqref{lowerorder} 
in the following way:
\begin{align*}
\left\lVert \frac{(a-\tilde{a})\left\langle b,c\right\rangle }{\vert d\vert^j}
\right\rVert_{C^{\frac{\alpha}{4},\alpha}}
\leq \tilde{C}(M,C)\Vert a-\tilde{a} \Vert_{C^{\frac{\alpha}{4},\alpha}}
\leq \tilde{C}(M,C) T^{\varepsilon}\Vert \gamma-\tilde{\gamma}\Vert_{\mathbb{E}_T}\,.
\end{align*}
The second and the third term of~\eqref{lowerorder}  
can be estimated similarly using H\"{o}lder inequality.

Considering the last term of~\eqref{lowerorder} we get
with the help of the identities
\begin{align*}
\frac{1}{\vert d \vert^2} -\frac{1}{\vert \tilde{d} \vert^2}=&
\left(\vert \tilde{d} \vert-\vert d \vert \right)
\left(\frac{1}{\vert d \vert^2\vert \tilde{d} \vert}
+\frac{1}{\vert d \vert\vert \tilde{d} \vert^2}\right)\,,\\
\frac{1}{\vert d \vert^j} -\frac{1}{\vert \tilde{d} \vert^j}=&
\left(\vert \tilde{d} \vert-\vert d \vert \right)
\left(\frac{1}{\vert d \vert^2\vert \tilde{d} \vert}
+\frac{1}{\vert d \vert\vert \tilde{d} \vert^2}\right)
\left( \frac{1}{\vert d \vert^2}
+\frac{1}{\vert \tilde{d} \vert^2}\right)
\left( \frac{1}{\vert d \vert^{j-4}}
+\frac{1}{\vert \tilde{d} \vert^{j-4}}\right)
\,,
\end{align*}
for $j\in\{6,8\}$ and Lemmata~\ref{regularityest} and~\ref{notproved}
that
$$
\left\lVert
\left(\frac{1}{\vert d\vert^j}-\frac{1}{\vert \tilde{d}\vert^j}\right)
\tilde{a}\left\langle \tilde{b},\tilde{c}\right\rangle\right\rVert_{C^{\frac{\alpha}{4},\alpha}}
\leq \tilde{C}(M,C)T^{\varepsilon}\Vert d-\tilde{d}\Vert_{C^{\frac{\alpha}{4},\alpha}}
\leq \tilde{C}(M,C) T^{\varepsilon}\Vert \gamma-\tilde{\gamma}\Vert_{\mathbb{E}_T}\,.
$$
Then we conclude that
$$
\Vert f(\gamma)-f(\tilde{\gamma})\Vert_{C^{\frac{\alpha}{4}},\alpha}\leq 
\tilde{C}(M,C) T^{\varepsilon}\Vert \gamma-\tilde{\gamma}\Vert_{\mathbb{E}_T}\,.
$$

\textbf{\textit{Step 2: Estimate on 
$\Vert b(\gamma)-b(\tilde{\gamma})\Vert_{C^{\frac{1+\alpha}{4}}}$}}

We now consider the boundary terms. We remind that we use the identification
$$
C^{\frac{1+\alpha}{4},1+\alpha}([0,T]\times\{0,1\}; \mathbb{R}^2)\backsimeq
C^{\frac{1+\alpha}{4}}([0,T];\mathbb{R}^4)
$$
introduced in~\ref{parabolicspaces}.

The expression $b(\gamma)-b(\tilde{\gamma})$ is given by 
\begin{align}\label{expressionb}
&\sum_{i=1}^3\frac{1}{\vert \varphi^i_x\vert^3}
\left\langle (\tilde{\gamma}^i_{xxx}-\gamma^i_{xxx}),\nu_0^i\right\rangle \nu_0^i
+\sum_{i=1}^3\frac{1}{\vert \gamma^i_x\vert^3}
\left\langle \gamma^i_{xxx},\nu^i\right\rangle \nu^i\nonumber \\
-&\sum_{i=1}^3\frac{1}{\vert \tilde{\gamma}^i_x\vert^3}
\left\langle \tilde{\gamma}^i_{xxx},\tilde{\nu}^i\right\rangle \tilde{\nu}^i 
-\frac{\mu}{2}\sum_{i=1}^{3}\left(\frac{\gamma^i_x}{\vert\gamma^i_x\vert}-
\frac{\tilde{\gamma}^i_x}{\vert\tilde{\gamma}^i_x\vert}
\right).
\end{align}
Writing the fourth term of~\eqref{expressionb} as
\begin{equation}\label{thirdterm}
-\frac{\mu}{2}\sum_{i=1}^{3}
\frac{\gamma^i_x-\tilde{\gamma}^i_x}{\vert\gamma^i_x\vert}+\left(
\frac{1}{\vert \gamma^i_x\vert}-
\frac{1}{\vert \tilde{\gamma}^i_x\vert}\right)\tilde{\gamma}^i_x\,,
\end{equation}
we find,
combining Lemma~\ref{notprovedbis} with~\eqref{usefulcone},
that
\begin{align*}
\left\lVert
\frac{\gamma^i_x-\tilde{\gamma}^i_x}{\vert\gamma^i_x\vert}
\right\rVert_{C^{\frac{1+\alpha}{4}}}
&\leq \tilde{C}(M,C)T^{\varepsilon}\Vert \gamma-\tilde{\gamma}\Vert_{\mathbb{E}_T}\,,\\
\left\lVert
\left(
\frac{1}{\vert \gamma^i_x\vert}-
\frac{1}{\vert \tilde{\gamma}^i_x\vert}\right)\tilde{\gamma}^i_x
\right\rVert_{C^{\frac{1+\alpha}{4}}}
&\leq 
C
\left\lVert
\frac{\vert \gamma^i_x-\tilde{\gamma}^i_x \vert}{\vert \gamma^i_x\vert \vert\tilde{\gamma}^i_x\vert}
\right\rVert_{C^{\frac{1+\alpha}{4}}}
\left\lVert
\tilde{\gamma}^i_x
\right\rVert_{C^{\frac{1+\alpha}{4}}}
\leq \tilde{C}(M,C)T^{\varepsilon}\Vert \gamma-\tilde{\gamma}\Vert_{\mathbb{E}_T}\,.
\end{align*}
Adding and subtracting
\begin{align*}
&\sum_{i=1}^{3}\frac{1}{\vert\varphi_x^i\vert^3}\langle\tilde{\gamma}^i_{xxx}-\gamma^i_{xxx},\nu^i\rangle\nu^i_0 + \sum_{i=1}^{3}\frac{1}{\vert\varphi_x^i\vert^3}\langle\tilde{\gamma}^i_{xxx}-\gamma^i_{xxx},\nu^i\rangle\nu^i
+\sum_{i=1}^3\frac{1}{\vert \gamma^i_x\vert^3}
\left\langle \tilde{\gamma}^i_{xxx},\nu^i\right\rangle \nu^i \\
+&
\sum_{i=1}^3\frac{1}{\vert \gamma^i_x\vert^3}
\left\langle \tilde{\gamma}^i_{xxx},\tilde{\nu}^i\right\rangle \nu^i 
+ \sum_{i=1}^3\frac{1}{\vert \gamma^i_x\vert^3}
\left\langle \tilde{\gamma}^i_{xxx},\tilde{\nu}^i\right\rangle \tilde{\nu}^i 
\end{align*}
to the remaining terms of~\eqref{expressionb}
we obtain
\begin{align}
&\sum_{i=1}^3\frac{1}{\vert \gamma^i_x\vert^3}
\left\langle \tilde{\gamma}^i_{xxx},\nu^i-\tilde{\nu}^i\right\rangle \nu^i
+\sum_{i=1}^3\frac{1}{\vert \gamma^i_x\vert^3}
\left\langle \tilde{\gamma}^i_{xxx},\tilde{\nu}^i\right\rangle\left(\nu^i- \tilde{\nu}^i \right)\label{lineone}\\
+&\sum_{i=1}^3\frac{1}{\vert \varphi^i_x\vert^3}
\left\langle (\tilde{\gamma}^i_{xxx}-\gamma^i_{xxx}),\nu_0^i-\nu^i\right\rangle \nu_0^i+\sum_{i=1}^3\frac{1}{\vert \varphi^i_x\vert^3}
\left\langle (\tilde{\gamma}^i_{xxx}-\gamma^i_{xxx}),\nu^i\right\rangle \left(\nu_0^i-\nu^i\right)\label{linetwo}\\
+&\sum_{i=1}^3
\left(\frac{1}{\vert \gamma^i_x\vert^3}-\frac{1}{\vert \tilde{\gamma}^i_x\vert^3}\right)
\left\langle \tilde{\gamma}^i_{xxx},\tilde{\nu}^i\right\rangle \tilde{\nu}^i 
+\sum_{i=1}^3\left(\frac{1}{\vert \varphi^i_x\vert^3}-\frac{1}{\vert \gamma^i_x\vert^3}\right)
\left\langle \tilde{\gamma}^i_{xxx}-\gamma^i_{xxx},\nu^i\right\rangle \nu^i 
\,.\label{linethree}
\end{align}

To estimate~\eqref{lineone} we use
\begin{align*}
&\left\lVert\nu^i-\tilde{\nu}^i
\right\rVert_{C^{\frac{1+\alpha}{4}}}=
\left\lVert
R\left(\frac{\gamma^i_x}{\vert \gamma^i_x\vert}\right)-R\left(
\frac{\tilde{\gamma}^i_x}{\vert \tilde{\gamma}^i_x\vert}\right)
\right\rVert_{C^{\frac{1+\alpha}{4}}}
\leq 
c
\left\lVert
\frac{\gamma^i_x}{\vert \gamma^i_x\vert}-
\frac{\tilde{\gamma}^i_x}{\vert \tilde{\gamma}^i_x\vert}
\right\rVert_{C^{\frac{1+\alpha}{4}}}\\
&
\leq 
 \tilde{C}(M,C)T^\varepsilon\Vert \gamma-\tilde{\gamma}\Vert_{\mathbb{E}_T}\,.
\end{align*}
To gain the last inequality we have repeated the argument
done for the term~\eqref{thirdterm}.
 
Combining 
$\left\lVert\tilde{\gamma}_{xxx}-\gamma_{xxx}\right\rVert_{C^{\frac{1+\alpha}{4}}}
\leq \Vert\tilde{\gamma}-\gamma \Vert_{\mathbb{E}_T}$
with
\begin{equation}
\left\lVert\nu_0^i-\nu^i
\right\rVert_{C^{\frac{1+\alpha}{4}}}
\leq 
 \overline{C}(M,C)T^\varepsilon\Vert \varphi-\gamma\Vert_{\mathbb{E}_T}
 \leq  \tilde{C}(M,C)T^\varepsilon\,
\end{equation}
the expression of~\eqref{linetwo} is bounded by 
$\tilde{C}(M,C)T^{\varepsilon}\Vert \gamma-\tilde{\gamma}\Vert_{\mathbb{E}_T}$.

For the terms in~\eqref{linethree} we use the equality
$$
\left(\frac{1}{\vert a\vert^3}-\frac{1}{\vert b\vert^3}\right)
=
\left(\frac{1}{\vert a\vert}-\frac{1}{\vert b\vert}\right)
\left(
\frac{1}{\vert a\vert^2}+\frac{1}{\vert a\vert\vert b\vert}
+\frac{1}{\vert b\vert^2}
\right)
$$
with $a,b\in\mathbb{R}^2$.
In particular we have
$$
\left\lVert
\left(\frac{1}{\vert \gamma^i_x\vert^3}-\frac{1}{\vert \tilde{\gamma}^i_x\vert^3}\right)
\right\rVert_{C^{\frac{1+\alpha}{4}}}
 \leq 
\overline{C}(M, C)
\left\lVert
\frac{\vert \gamma^i_x-\tilde{\gamma}^i_x \vert}{\vert \gamma^i_x\vert \vert\tilde{\gamma}^i_x\vert}
\right\rVert_{C^{\frac{1+\alpha}{4}}}
\leq \tilde{C}(M,C)T^{\varepsilon}\Vert \gamma-\tilde{\gamma}\Vert_{\mathbb{E}_T}\,.
$$
Also the last term in~\eqref{linethree} is bounded from above by
$\tilde{C}(M,C)T^{\varepsilon}\Vert \gamma-\tilde{\gamma}\Vert_{\mathbb{E}_T}$
using 
$$
\left\lVert
\left(\frac{1}{\vert \varphi^i\vert^3}-\frac{1}{\vert \gamma^i_x\vert^3}\right)
\right\rVert_{C^{\frac{1+\alpha}{4}}}
\leq \overline{C}(M,C)T^{\varepsilon}\Vert \varphi-\gamma\Vert_{\mathbb{E}_T}
\leq \tilde{C}(M,C)T^{\varepsilon}\,
$$
and
$\left\lVert\tilde{\gamma}^i_{xxx}-\gamma^i_{xxx}\right\rVert_{C^{\frac{1+\alpha}{4}}}
\leq \Vert\tilde{\gamma}-\gamma \Vert_{\mathbb{E}_T}$.

Hence we have shown that
$$
\Vert b(\gamma)-b(\tilde{\gamma})\Vert_{C^{\frac{1+\alpha}{4}}}
\leq \tilde{C}(M,C)T^{\varepsilon}\Vert \gamma-\tilde{\gamma}\Vert_{\mathbb{E}_T}\,.
$$ 

{\textbf{\textit{Step 3: Contraction estimate}}}

The estimates presented in {\textbf{\textit{Step 1}}} and {\textbf{\textit{Step 2}}}
imply for all $T\in (0,\min\{\tilde{t}(M,C),T_0\}]$
\begin{align*}
\Vert K_T(\gamma)-K_T(\tilde{\gamma})\Vert_{\mathbb{E}_T}
&=\Vert L^{-1}_T(N_T(\gamma))-L^{-1}(N_T(\tilde{\gamma}))\Vert_{\mathbb{E}_T}\\
&=\Vert L^{-1}_T(N_T(\gamma)-N_T(\tilde{\gamma}))\Vert_{\mathbb{E}_T}\\
&\leq \sup_{T\in[0,T_0]}\Vert L^{-1}_T\Vert_{\mathcal{L}(\mathbb{F}_T,\mathbb{E}_T)}
\Vert N_T(\gamma)-N_T(\tilde{\gamma})\Vert_{\mathbb{F}_T}\\
&\leq \tilde{C}(M,C,R,T_0)T^\varepsilon\Vert \gamma-\tilde{\gamma}\Vert_{\mathbb{E}_T}\,
\end{align*}
with $0<\varepsilon<1$.
Choosing $T(M)
\in(0,\min\{\tilde{t}(M,C),T_0\}]
$ small enough we can conclude that
$K_T$ is a contraction for every $T\in (0,T(M)]$.
\end{proof}

\subsubsection{Short time existence Theorem}

\begin{prop}\label{welldefined}
There exists a positive radius $M=M(R,C,T_0)$ and a positive time $\tilde{T}(M)$ 
such that for all $T\in (0,\tilde{T}(M)]$
the map
$K_T:\mathbb{E}^\varphi_T\cap \overline{B_{M}}\to\mathbb{E}^\varphi_T\cap \overline{B_{M}}$ 
is  well defined.
\end{prop}

\begin{proof}
	We define 
	\[M:=2c(T_0)\max\{R,\Vert N_{T_0}(\varphi)\Vert_{\mathbb{F}_{T_0}}\}\] 
	and note that the quantity  $\Vert N_{T_0}(\varphi)\Vert_{\mathbb{F}_{T_0}}$ does not depend on $T_0$ but only on $R=\Vert\varphi\Vert_{C^{4+\alpha}([0,1])}$ and $C$. In particular we have for any $T\in (0,T_0]$
	\begin{equation}\label{estimatephi}
	\Vert K_T(\varphi)\Vert_{\mathbb{E}_T}\leq\sup_{T\in (0,T_0]}\Vert L_T^{-1}\Vert_{\mathcal{L}(\mathbb{F}_T,\mathbb{E}_T)}\Vert N_T(\varphi)\Vert_{\mathbb{F}_T}\leq c(T_0)\Vert N_{T_0}(\varphi)\Vert_{\mathbb{F}_{T_0}}\leq \frac{M}{2}\,.
	\end{equation}
	Let $T(M)$ be the corresponding time as in Proposition~\ref{contraction} such that for all $T\in (0,T(M)]$ the map $K_T$ is a contraction on $\overline{B_M}$
	and let $\gamma\in \mathbb{E}^\varphi_T\cap \overline{B_{M}}$. As $\varphi$ lies as well in $\overline{B_M}$ we have for all $T\in (0,T(M)]$ that 
	\begin{equation*}
	\Vert K_T(\gamma)-K_T(\varphi)\Vert_{\mathbb{E}_T}\leq \tilde{C}(M,C,R,T_0)T^\varepsilon\Vert \gamma-\varphi\Vert_{\mathbb{E}_T}\leq \tilde{C}(M,C,R,T_0)2MT^\varepsilon
	\end{equation*}
	for some $0<\varepsilon<1$.
	We choose $\tilde{T}(M)\in (0,T(M)]$ so small that $\tilde{C}(M,C,R,T_0)2MT^\varepsilon\leq \frac{M}{2}$
	for all $T\in (0,\tilde{T}(M)]$. Combining this with~\eqref{estimatephi} implies for all $T\in (0,\tilde{T}(M)]$ and $\gamma\in \mathbb{E}_T^\varphi\cap \overline{B_M}$ that
	\begin{equation*}
	\Vert K_T(\gamma)\Vert_{\mathbb{E}_T}\leq \Vert K_T(\varphi)\Vert_{\mathbb{E}_T}+\Vert K_T(\gamma)-K_T(\varphi)\Vert_{\mathbb{E}_T}\leq M.
	\end{equation*}
	
\end{proof}
\begin{teo}\label{existenceanalyticprob}
Let $\varphi$ be an admissible initial parametrisation. There exists a positive radius $M$ and a positive time $T$ such that the system~\eqref{thetaC0}
has a unique solution in $C^{\frac{4+\alpha}{4},4+\alpha}\left([0,T]\times[0,1]\right)\cap \overline{B_M}$.	
\end{teo}
\begin{proof}
Let $M$ and $\tilde{T}(M)$ be the radius and time as in Proposition~\ref{welldefined} and let $T\in (0,\tilde{T}(M)]$. Then
\[K_T:\mathbb{E}_T^\varphi\cap \overline{B_M}\to\mathbb{E}_T^\varphi\cap \overline{B_M}\]
is a contraction of the complete metric space $\mathbb{E}_T^\varphi\cap\overline{B_M}$ and has a unique fixed point in $\mathbb{E}_T^\varphi\cap \overline{B_M}$ by the Contraction Mapping Principle. As the solutions of~\eqref{thetaC0} in $C^{\frac{4+\alpha}{4},4+\alpha}\left([0,T]\times[0,1]\right)\cap \overline{B_M}$ are precisely the fixed points of $K_T$ in $\mathbb{E}_T^\varphi\cap \overline{B_M}$, the claim follows.
\end{proof}

\subsection{Geometric evolution and geometric uniqueness}\label{geometricstuff}

In the previous subsection we have shown that  there exists a unique solution 
to the Analytic Problem provided that the initial data is admissible.
It remains to establish a relation between geometrically admissible initial
networks and admissible initial parametrisations. 
This is necessary  to solve the question of geometric existence and uniqueness of the flow.

\begin{lemma}\label{repara}
Suppose that $\Theta_0$ is a geometrically admissible initial network in the sense of
Definition~\ref{admg} parametrised by 
$\sigma^i$. Then there exist
three smooth functions $\theta^i:[0,1]\to [0,1]$ such that the reparametrisation 
$\left(\sigma^i\circ\theta^i\right)$
of  $\Theta_0$ by $(\theta^i)$ is an admissible initial parametrisation for the Analytic Problem~\eqref{thetaC0}.
\end{lemma}
\begin{proof}
Without loss of generality we may assume that $\sigma^1(0)=\sigma^2(0)=\sigma^3(0)$
and $\sigma^1(1)=\sigma^2(1)=\sigma^3(1)$. 

We look for some smooth maps $\theta^i:[0,1]\to [0,1]$ such that 
 $\theta^i_x(x)\neq 0$
for every $x\in [0,1]$,  $\theta^i(0)=0$ and $\theta^i(1)=1$.
Then $\varphi^i:=\sigma^i\circ\theta^i:[0,1]\to\mathbb{R}^2$ is regular and
of class $C^{4+\alpha}([0,1])$ and the concurrency condition is satisfied.
Since the unit normal vectors are invariant under reparametrisation, we have
$\mathrm{span}\{\nu^1_0,\nu^2_0,\nu^3_0\}=\mathbb{R}^2$ at both triple junctions. Also
the curvature and third order condition remain
true being invariant under reparametrisation.
In order to fulfil the second order condition $\varphi^i_{xx}(0)=\varphi^i_{xx}(1)=0$
we demand that $\theta_x^i(y)=1$ and
 $\theta^i_{xx}(y)=\frac{-\sigma^i_{xx}(y)}{\sigma^i_x(y)}$
 for $y\in\{0,1\}$.
Indicating with subscript $\varphi$ 
the quantities computed for the reparametrised object
the last property we have to prove is that at $y\in\{0,1\}$
it holds
$$
A^1_\varphi\nu^1_\varphi+T^1_\varphi\tau^1_\varphi=
A^2_\varphi\nu^2_\varphi+T^2_\varphi\tau^2_\varphi=
A^3_\varphi\nu^3_\varphi+T^3_\varphi\tau^3_\varphi\,.
$$
Since the geometric quantities are 
 invariant under reparametrisation, this 
is equivalent to 
\begin{equation}\label{compatib}
A^1_0\nu^1_0+T^1_\varphi\tau^1_0=
A^2_0\nu^2_0+T^2_\varphi\tau^2_0=
A^3_0\nu^3_0+T^3_\varphi\tau^3_0\,.
\end{equation}
This is satisfied if and only if it holds
\begin{align*}
&\sin(\alpha^1) A_0^1+\sin(\alpha^2) A_0^2+\sin(\alpha^3) A_0^3=0\,,\\
&\sin(\alpha^1)T_\varphi^1=\cos(\alpha^2)A_0^2-\cos(\alpha^3)A_0^3\,,\\
&\sin(\alpha^2)T_\varphi^2=\cos(\alpha^3)A_0^3-\cos(\alpha^1)A_0^1\,,\\
&\sin(\alpha^3)T_\varphi^3=\cos(\alpha^1)A_0^1-\cos(\gamma^2)A_0^2\,.
\end{align*}
The first equation is fulfilled as $\Theta_0$ is a geometrically
admissible initial network. Asking that
$\theta_{xxx}^i(0)=1=\theta_{xxx}^i(1)$ the remaining three equations are of the form 
$$
a^i(\sigma^i_{x})(y)\,\theta^i_{xxxx}(y)
+b^i(\sigma^i_{xxxx}, \sigma^i_{xxx}, \sigma^i_{xx}, \sigma^i_{x})(y)=0\,,
$$
where $a^i$ and $b^i$ are given non--linear functions.
Hence $\theta^i_{xxxx}(0)$ and $\theta^i_{xxxx}(1)$ are uniquely determined
for every $i\in\{1,2,3\}$.
Thus we may choose $\theta^i$ to be the fourth Taylor polynomial
near each triple junction determined by the values of the derivatives at the triple junctions
that we have just derived. 
\end{proof}

\begin{teo}\label{geomexistence}[Geometric existence and uniqueness]
Let $\Theta_0$ be a geometrically admissible initial network. Then there exists
a positive time $T$ such that within the time interval $[0,T]$ the
Problem~\ref{geometricproblem} admits 
a unique solution $\Theta(t)$ 
 in the sense of 
Definition~\ref{geometricsolution}.
\end{teo}
\begin{proof}
Let $\Theta_0$  be a geometrically admissible initial network.
Thanks to Lemma~\ref{repara} 
there exists an admissible initial parametrisation $\varphi$ for
$\Theta_0$.
Then by Theorem~\ref{existenceanalyticprob} there exists a solution 
$\tilde{\gamma}$ to system~\eqref{thetaC0} on some time interval $[0,\tilde{T}]$.
In particular the family of networks $(\tilde{\Theta}(t))_{t\in[0,\tilde{T}]}$
with $\tilde{\Theta}(t)=\cup_{i=1}^3\tilde{\gamma}^i(t)$ is a solution
to Problem~\ref{geometricproblem} 
 in the sense of Definition~\ref{geometricsolution}.

\medskip

Suppose that there exists another solution $(\Theta(t))_{t\in[0,T]}$ 
to the $C^0$--flow in the sense of Definition~\ref{geometricsolution}
in a certain time interval $[0,T]$ with $T>0$
and with the same initial network $\Theta_0$.
We assume that $\left(\Theta(t)\right)$ is parametrised by $\gamma$.
\medskip

It remains to prove that there exists a $\overline{T}\in (0,\min\{\tilde{T},T\}]$
such that 
 the networks $(\tilde{\Theta}(t))$ and $(\Theta(t))$
coincide for every $t\in[0,\overline{T}]$.
To show this it is enough to find
a time dependent family of reparametrisations $\psi:[0,\overline{T}]\times[0,1]\to[0,1]$
such that $\tilde{\gamma}(t,x)=\gamma(t,\psi(t,x))$ for all $(t,x)\in[0,\overline{T}]\times[0,1]$.

Following the approach presented in~\cite[Theorem 5.3]{garckenov} to prove
the existence of $\psi$ it is sufficient to solve the following initial boundary value problem
on $[0,\overline{T}]\times[0,1]$:
\begin{equation}
\begin{cases}
\psi_t(t,x)+\frac{2\psi_{xxxx}(t,x)}{\vert\gamma_x(t,\psi(t,x))\vert^4 \vert \psi_x(t,x)\vert^4}
+g(\psi_{xxx}, \psi_{xx}, \psi_x, \psi, \gamma_{xxxx}, \gamma_{xxx},\gamma_{xx}, \gamma_x,
\gamma_t)=0,\\
\psi(t,y)=0,\\
\psi_{xx}(t,y)=-\frac{\gamma_{xx}(t,y)\psi^2_x(t,y)}{\gamma_x(t,y)},\\
\gamma(0,\psi(0,x))=\varphi(x)
\end{cases}
\end{equation}
with $y\in\{0,1\}$ and $(t,x)\in [0,\overline{T}]\times [0,1]$.

\medskip
Indeed assume that $\psi$ exists. Then the family of parametrisations $Z(t,x):=\gamma(t,\psi(t,x))$ is a solution to the Analytic Problem~\eqref{thetaC0}. The motion equation follows from the PDE satisfied by $\psi$ and the geometric evolution of $\Theta$ in normal direction. The geometric boundary conditions (concurrency, curvature and third order) are satisfied as $\Theta$ is a solution to the Geometric Problem. The boundary condition on $\psi_{xx}$ guarantees that $Z$ satisfies the second order condition. By uniqueness of the Analytic Problem for short time $Z$ and $\tilde{\gamma}$ need to coincide on a small time interval.

\medskip

We note that the parabolic  boundary value problem for 
$\psi^i$ has a similar structure as the system of PDEs we studied before.
The associated linearised system satisfies the Lopatinskii--Shapiro condition and a contraction argument allows us to conclude that the desired function $\psi^i$ exists and is unique.
\end{proof}

\section{$C^1$--Flow}\label{C1flow}

In this section we briefly discuss an 
 analogous  result to  Theorem~\ref{geomexistence}
in the case of a $C^1$ Triod evolving with respect 
to the gradient flow of the elastic energy $E_\mu$. 

\medskip

In contrast to the $C^0$--case we now demand that at the triple junction 
the curves form an angle of $120$ degrees. Moreover we ask that 
the other three endpoints $P^i$ remain fixed on $\partial\Omega$ during the evolution. This corresponds to the Navier case.

\begin{dfnz}\label{admg1}
A Triod $\mathbb{T}_0$ is called geometrically admissible for
Problem~\ref{generalgeometricproblemII} if
\begin{itemize}
	\item the curves meet in one triple junction
forming angles of $120$ degrees 
and satisfying $\sum_{i=1}^{3}k^i_0=0$ 
and $\sum_{i=1}^{3}\left(2k_{0,s}^{i}\nu_0^{i}-(k^i_0)^2\tau^{i}_0\right)=0$;
	\item each curve has zero curvature at the endpoint $P^i$ on $\partial\Omega$;
	\item at the triple junction it holds that 
	$A^1_0+A^2_0+A^3_0=0$ where $A^i_0=2k_{0,ss}^{i}+\left(k_0^{i}\right)^{3}-\mu k_0^{i}$;
	\item there exists a parametrisation $\gamma=(\gamma^1,\gamma^2,\gamma^3)$ 
	of $\mathbb{T}_0$ 
	such that every curve is regular and an element of $C^{4+\alpha}([0,1];\mathbb{R}^2)$.
\end{itemize}
\end{dfnz}
\begin{dfnz}\label{geometricsolution1}
	Let $\mathbb{T}_0$ be a geometrically admissible initial Triod.
	A time dependent family of Triods $\mathbb{T}(t)$
	is a solution to
	the  $C^1$--Flow with initial data $\mathbb{T}_0$ 
	in a time interval $[0,T]$ with $T>0$
	if and only if there exist
	parametrisations $\gamma(t)=(\gamma^1(t),\gamma^2(t),\gamma^3(t))$ 
	of $\mathbb{T}(t)$, $t\in[0,T]$, with
	$\gamma^i\in C^{\frac{4+\alpha}{4}, 4+\alpha}([0,T]\times[0,1],\mathbb{R}^2)$
	such that 
	for every  $t\in\left[0,T\right],x\in\left[0,1\right]$ and for $i\in\{1,2,3\}$
	the following system
	\begin{equation}
	\begin{cases}
	\begin{array}{lll}
	\left\langle \gamma_t^{i}(t,x), \nu^i(t,x)\right\rangle  \nu^i(t,x)
	=-A^{i}(t,x)\nu^{i}(t,x)& &\text{motion,}\\
	\gamma^{1}\left(t,0\right)=\gamma^{2}\left(t,0\right)=\gamma^{3}\left(t,0\right)&
	&\text{concurrency condition,}\\
	\tau^1(t,0)+\tau^2(t,0)+\tau^3(t,0)=0 & &\text{angle condition,}\\
	k^{1}(t,0)+k^2(t,0)+k^3(t,0)=0 & &\text{curvature condition,}\\
	\sum_{i=1}^{3}\left(2k_{s}^{i}\nu^{i}-(k^i)^2\tau^{i}\right)(t,0)=0& &\text{third order condition,}\\
	\gamma^i(t,1)=P^i& &\text{fixed endpoints,}\\
	k^{i}(t,1)=0 & &\text{curvature condition,}\\
	\mathbb{T}(0)=\mathbb{T}_0& &
	\text{initial data}
	\end{array}
	\end{cases}
	\end{equation}
	is satisfied.
\end{dfnz}

Our aim is to prove geometric existence and uniqueness of the $C^1$--flow.

\begin{teo}\label{geoexistence2}
Let $\mathbb{T}_0$ be a geometrically admissible initial Triod. Then there exists
a positive time $T$ such that within the time interval $[0,T]$ the
Problem~\ref{generalgeometricproblemII} admits 
a unique solution $(\mathbb{T}(t))$ 
 in the sense of 
Definition~\ref{geometricsolution1}.
\end{teo}

\subsection{Proof of Theorem~\ref{geoexistence2}}

Similar arguments as used in the previous section on geometric existence and uniqueness 
justify that to find a unique geometric solution to the $C^1$--flow of a Triod
we may focus on the Analytic Problem formulated 
for an appropriate time dependent parametrisation of the flow. 
To this end we now sketch how to prove a short time
 existence result for the associated system of PDEs.

We again choose the same tangential velocity and we 
impose the tangential second order condition 
at both boundary points which leads to the following definition of admissible initial
parametrisation and to the subsequent problem.

\begin{dfnz}\label{adm1}
	A parametrisation $\varphi$
	of an initial Triod $\mathbb{T}_0$
	is admissible 
	if 
\begin{itemize}
\item each curve $\varphi^i$ is regular and in $C^{4+\alpha}([0,1])$ for some $\alpha\in(0,1)$;
\item at the triple junction $O=\gamma^i(0)$ the curves satisfy the concurrency, 
angle and tangential second order condition as well as $\sum_{i=1}^3 k^i_0=0$, $\sum_{i=1}^{3}\left(2k_{0,s}^{i}\nu_0^{i}-(k^i_0)^2\tau^{i}_0\right)=0$ and 
		\begin{equation*}
		A_0^{i}\nu_0^{i}+T_0^{i}\tau_0^{i}
		=A_0^{j}\nu_0^{j}+T_0^{j}\tau_0^{j}
		\end{equation*}
		for $i,j\in\{1,2,3\}$ 
		where $A^i_0$ and $T^i_0$ denote the equations~\eqref{Apart} and~\eqref{Tang}
		applied to $\varphi^i$;
		\item at the endpoints $P^i=\gamma^i(1)$ it holds that $\gamma^i_{xx}=0$.
	\end{itemize}
\end{dfnz}

\begin{prob}\label{analyticprob1}
	For every  $t\in\left[0,T\right),x\in\left[0,1\right]$ and for $i\in\{1,2,3\}$
	\begin{equation}\label{TriodC1}
	\begin{cases}
	\begin{array}{ll}
	\gamma_t^{i}(t,x)=-A^{i}(t,x)\nu^{i}(t,x)-T^{i}(t,x)\tau^{i}(t,x) &\text{motion,}\\
	\gamma^{1}\left(t,0\right)=\gamma^{2}\left(t,0\right)=\gamma^{3}\left(t,0\right)
	&\text{concurrency condition,}\\
	\tau^1(t,0)+\tau^2(t,0)+\tau^3(t,0)=0  &\text{angle condition,}\\
	k^{1}(t,0)+k^2(t,0)+k^3(t,0)=0  &\text{curvature condition,}\\
	\left\langle\gamma^i_{xx}(t,0),\tau^i(t,0) \right\rangle=0  
	&\text{tangential second order condition,}\\
	\sum_{i=1}^{3}\left(2k_{s}^{i}\nu^{i}-(k^i)^2\tau^{i}\right)(t,0)=0 &\text{third order condition,}\\
	\gamma^i(t,1)=\varphi^i(1) &\text{fixed endpoints,}\\
	\gamma_{xx}^{i}(t,1)=0 &\text{second order condition,}\\
	\gamma^i(0,x)=\varphi^i(x)  &
	\text{initial data}
	\end{array}
	\end{cases}
	\end{equation}
	with $\varphi^i$ admissible initial parametrisation as defined in~\ref{adm1}.
\end{prob}

\begin{teo}\label{existenceanalyticprobC1}
Let $\varphi$ be an admissible initial parametrisation. 
There exists a positive radius $M$ and a positive time $T$ 
such that the system~\eqref{TriodC1}
has a unique solution in $C^{\frac{4+\alpha}{4},4+\alpha}\left([0,T]\times[0,1]\right)\cap \overline{B_M}$.	
\end{teo}

To prove existence of a unique solution 
$\gamma\in C^{\frac{4+\alpha}{4},4+\alpha}\left([0,T]\times[0,1]\right)$ to~\eqref{TriodC1} 
one linearises the system around the initial data, proves existence of the linear system and provides contraction estimates which guarantee the existence of a solution to the original system. 
As many arguments are exactly analogous to the case of $C^0$--Theta networks, 
we focus on  the relevant changes. 

\medskip

Hence we prove that the
Lopatinskii--Shapiro condition is fulfilled and
we describe the 
contraction estimates for the boundary terms that are remarkably different.

\subsubsection{Linearisation}

We state here the linearisation of the  remaining boundary conditions that appear in the system~\ref{TriodC1} (and do not in system~\ref{thetaC0}).

Notice that all the conditions at $x=1$ are already linear.
We linearise the conditions at $x=0$.
The angle condition becomes
\begin{align}
-\sum_{i=1}^3 \left(\frac{\gamma^i_x}{\vert \varphi^i_x\vert}
-\frac{\varphi^i_x\left\langle
\gamma^i_x,\varphi^i_x\right\rangle}{\vert \varphi^i_x\vert^3}\right)
&=\sum_{i=1}^3 \left(\left(\frac{1}{\vert \gamma^i_x\vert}
 -\frac{1}{\vert\varphi^i_x\vert}\right)\gamma^i_x +
 \frac{\varphi^i_x\left\langle \gamma^i_x,\varphi^i_x\right\rangle}{\vert \varphi^i_x\vert^3}
\right)
=:b^1(\gamma) \,.
\end{align}
Considering   the unit tangent vector $\tau_0^i$ and normal vector $\nu^i_0$ at time zero,
\begin{equation}
\tau^i_0=\frac{1}{\vert \varphi^i_x\vert}\begin{pmatrix}
\tilde{\varphi}^i_x \\
\hat{\varphi}^i_x
\end{pmatrix}\quad\quad
\nu^i_0=\frac{1}{\vert \varphi^i_x\vert}\begin{pmatrix}
\hat{\varphi}^i_x \\
-\tilde{\varphi}^i_x\,,
\end{pmatrix}
\end{equation}
the linearised version of the curvature condition is of the form:
\begin{align}
\sum_{i=1}^3\frac{1}{\vert \varphi^i_x\vert^2}
\left\langle \gamma^i_{xx},\nu_0^i\right\rangle =
\sum_{i=1}^3\frac{1}{\vert \varphi^i_x\vert^2}
\left\langle \gamma^i_{xx},\nu_0^i\right\rangle 
- \sum_{i=1}^3\frac{1}{\vert \gamma^i_x\vert^2}
\left\langle \gamma^i_{xx},\nu^i\right\rangle =:b^2(\gamma,)
\,
\end{align}
and the linearised tangential second order condition becomes:
\begin{align*}
\left\langle \gamma^i_{xx},\tau^i_0\right\rangle= 
\left\langle \gamma^i_{xx},\tau^i_0-\tau^i \right\rangle =:b^{3,i}(\gamma^i)\,.
\end{align*}

As we are linearising only the highest order part of each condition,
the linearised third order condition that will appear in the linearised system 
for~\ref{TriodC1} is the same as in the case of $C^0$ Theta--networks:
\begin{align}
\sum_{i=1}^3\frac{1}{\vert \varphi^i_x\vert^3}
\left\langle \gamma^i_{xxx},\nu_0^i\right\rangle \nu_0^i=
\sum_{i=1}^3\frac{1}{\vert \varphi^i_x\vert^3}
\left\langle \gamma^i_{xxx},\nu_0^i\right\rangle \nu_0^i
-\sum_{i=1}^3\frac{1}{\vert \gamma^i_x\vert^3}
\left\langle \gamma^i_{xxx},\nu^i\right\rangle \nu^i 
-h^i(\gamma^i_{xx},\gamma^i_x)=:-b^4
\,.
\end{align}

Thus the linearised system around the initial data associated to~\ref{TriodC1} is
\begin{equation}\label{lynTriod}
\begin{cases}
\begin{array}{lll}
\gamma^i_t+\frac{2}{\vert\varphi^i_x\vert^4}\gamma^i_{xxxx}&=f^i &\text{motion,}\\
\gamma^{1}(t,0)-\gamma^{2}(t,0)&=0 &\text{concurrency,}\\
\gamma^{1}(t,0)-\gamma^{3}(t,0)&=0 &\text{concurrency,}\\
-\sum_{i=1}^3 \left(\frac{\gamma^i_x}{\vert \varphi^i_x\vert}
-\frac{\varphi^i_x\left\langle \gamma^i_x,\varphi^i_x\right\rangle}{\vert \varphi^i_x\vert^3}\right)&=b^1(t,0) &\text{angle condition,}\\
\sum_{i=1}^3\frac{1}{\vert \varphi^i_x(0)\vert^2}
\left\langle \gamma^i_{xx}(t,0),\nu_0^i(0)\right\rangle&=b^2(t,0) &\text{curvature condition,}\\
\left\langle \gamma^i_{xx}(t,0),\tau^i_0(0)\right\rangle&=b^{3,i}(t,0) &\text{tangential second order,}\\
-\sum_{i=1}^3\frac{1}{\vert \varphi^i_x(0)\vert^3}
\left\langle \gamma^i_{xxx}(t,0),\nu_0^i(0)\right\rangle \nu_0^i(0)&=b^4(t,0)  &\text{third order,}\\
\gamma^{i}\left(t,1\right)&=P^{i}  &\text{fixed endpoints,}\\
\gamma^i_{xx}(t,1)&=0 &\text{second order,}\\
\gamma^{i}(0)&=\psi^{i} &\text{initial condition}\\
\end{array}
\end{cases}
\end{equation}
for $i\in\{1,2,3\}$
where have omitted the dependence on $(t,x)\in[0,T]\times[0,1]$ in the  motion
equation and on $x\in[0,1]$ in the initial condition.\\

\subsubsection{Lopatinskii--Shapiro condition}

\begin{dfnz}
	Let $\lambda\in\mathbb{C}$ with $\Re(\lambda)>0$ be arbitrary.
	The Lopatinskii--Shapiro condition in $x=0$
	demands that every solution $(\gamma^i)_{i=1,2,3}\in C^4([0,\infty),(\mathbb{C}^2)^3)$ to
	the system of ODEs
	\begin{equation}\label{LopatinskiiShapiroinzero}
	\begin{cases}
	\begin{array}{llll}
	\lambda \gamma^i(x)+\frac{1}{\vert\varphi^i_x \vert^4}\gamma^i_{xxxx}(x)&=0&\;x\in[0,\infty), i\in\{1,2,3\}&\;\text{motion,}\\
	\gamma^{1}(0)-\gamma^{2}(0)&=0 &\; &\;\text{concurrency,}\\
	\gamma^{2}(0)-\gamma^{3}(0)&=0 &\; &\;\text{concurrency,}\\
	\sum_{i=1}^3 \frac{1}{\vert\varphi^i_x\vert}\langle\gamma^i_x(0),\nu^i_0\rangle\nu^i_0&=0 &\;&\;\text{angle condition,}\\
	\sum_{i=1}^3\frac{1}{\vert \varphi^i_x\vert^2}
	\left\langle \gamma^i_{xx}(0),\nu_0^i\right\rangle&=0&\;&\;\text{curvature condition,}\\
	\left\langle \gamma^i_{xx}(0),\tau^i_0\right\rangle&=0 &\;i\in\{1,2,3\}&\;\text{tangential second order,}\\
	\sum_{i=1}^3\frac{1}{\vert \varphi_x^i\vert ^3}
	\left\langle \gamma^i_{xxx}(0),\nu_0^i\right\rangle \nu_0^i&=0 &\; &\;\text{third order}\\
	\end{array}
	\end{cases}
	\end{equation}
	which satisfies $\lim_{x\to\infty}\lvert \gamma^i(x)\rvert=0$ has to be the trivial solution.
	In addition, the Lopatinskii--Shapiro condition in $x=1$
	demands that every solution $(\gamma^i)_{i=1,2,3}\in C^4([0,\infty),(\mathbb{C}^2)^3)$ to
	the system of ODEs
	\begin{equation}\label{LopatinskiiShapiroinone}
	\begin{cases}
	\begin{array}{llll}
	\lambda \gamma^i(x)+\frac{1}{\vert\varphi^i_x \vert^4}\gamma^i_{xxxx}(x)&=0&\;x\in[0,\infty), i\in\{1,2,3\}&\;\text{motion,}\\
	\gamma^{i}(0)&=0 &\;i\in\{1,2,3\}&\;\text{fixed endpoints,}\\
	\gamma^i_{xx}(0)&=0 &\;i\in\{1,2,3\}&\;\text{second order}\\
	\end{array}
	\end{cases}
	\end{equation}
	which satisfies $\lim_{x\to\infty}\lvert \gamma^i(x)\rvert=0$ has to be the trivial solution.
	Notice that we have omitted the dependence of $\varphi^i_x,\tau^i_0,\nu^i_0$
	on $x=0$ in the first system and $x=1$ in the second system, respectively.
\end{dfnz}

\begin{lemma}\label{proofofLS}
	The Lopatinskii--Shapiro conditions are satisfied.
\end{lemma}
\begin{proof}
	We show that the condition is satisfied at $x=0$.
	Let $\lambda\in\mathbb{C}$ with $\Re(\lambda)>0$  and let $\gamma$ be a solution to the system
	which satisfies $\lim_{x\to\infty}\lvert \gamma^i(x)\rvert=0$.
	Since $\varphi$ is an admissible initial parametrisation,  the angle condition $\sum_{i=1}^3 \tau_0^i=0$
	implies that it holds  $\sum_{i=1}^3 \nu^i_0=0$.
Combining this with the third order condition we have 
	\begin{equation}
	\frac{1}{\vert \varphi_x^1\vert ^3}
	\left\langle \gamma^1_{xxx}(0),\nu_0^1\right\rangle
	=\frac{1}{\vert \varphi_x^2\vert ^3}
	\left\langle \gamma^2_{xxx}(0),\nu_0^2\right\rangle
	=\frac{1}{\vert \varphi_x^3\vert ^3}
	\left\langle \gamma^3_{xxx}(0),\nu_0^3\right\rangle\,.\label{thirdorderidentity}
	\end{equation}
	Multiplying the motion equation by $\frac{1}{\vert\varphi^i_x\vert}\overline{\langle\gamma^i_{xx}(x),\nu^i_0\rangle}\nu^i_0$,
	summing and integrating we obtain
	\begin{align}
	0&=\lambda\sum_{i=1}^{3}\frac{1}{\vert\varphi^i_x\vert}\int_{0}^{\infty}\langle\gamma^i,\nu^i_0\rangle\overline{\langle\gamma^i_{xx},\nu^i_0\rangle}\mathrm{d}x+\sum_{i=1}^{3}\frac{1}{\vert\varphi^i_x\vert^5}\int_{0}^{\infty}\langle\gamma^i_{xxxx},\nu^i_0\rangle\overline{\langle\gamma^i_{xx},\nu^i_0\rangle}\mathrm{d}x\\
	&=-\lambda\sum_{i=1}^{3}\frac{1}{\vert\varphi^i_x\vert}\int_{0}^{\infty}\left\vert\langle\gamma^i_x,\nu^i_0\rangle\right\vert^2\mathrm{d}x-\sum_{i=1}^{3}\frac{1}{\vert\varphi^i_x\vert^5}\int_{0}^{\infty}\left\vert\langle\gamma^i_{xxx},\nu^i_0\rangle\right\vert^2\mathrm{d}x\\
	&+\lambda\sum_{i=1}^{3}\frac{1}{\vert\varphi^i_x\vert}\langle\gamma^i(0),\nu^i_0\rangle\overline{\langle\gamma^i_x(0),\nu^i_0\rangle}+\sum_{i=1}^{3}\frac{1}{\vert\varphi^i_x\vert^5}\langle\gamma^i_{xxx}(0),\nu^i_0\rangle\overline{\langle\gamma^i_{xx}(0),\nu^i_0\rangle}\,.\label{LopShap}
	\end{align}
The concurrency and the angle condition imply that the first boundary term vanishes:
	\begin{equation}
	\left\langle\gamma^1(0),\sum_{i=1}^{3}\frac{1}{\vert\varphi^i_x\vert}\overline{\langle\gamma^i_x(0),\nu^i_0\rangle}\nu^i_0\right\rangle=0\,.
	\end{equation}
Using~\eqref{thirdorderidentity} and the curvature condition we see that the other boundary term is given by	
\begin{equation}
\frac{1}{\vert \varphi_x^1\vert ^3}
\left\langle \gamma^1_{xxx}(0),\nu_0^1\right\rangle\sum_{i=1}^{3}\frac{1}{\vert\varphi^i_x\vert^2}\overline{\langle\gamma^i_{xx}(0),\nu^i_0\rangle}=0\,.
\end{equation}
Considering the real part of~\eqref{LopShap} we conclude that $\langle\gamma^i_x(x),\nu^i_0\rangle=0$ for all $x\geq 0$ which implies $\langle\gamma^i_{xxxx}(x),\nu^i_0\rangle=0$ for all $x\geq 0$. Using the motion equation and the fact that $\Re(\lambda)>0$ we obtain $\langle\gamma^i(x),\nu^i_0\rangle =0$ for all $x\geq 0$ and in particular $\langle\gamma^1(0),\sum_{i=1}^{3}a^i\nu_0^i\rangle=0$ for all $a\in\mathbb{R}^3$. As the vectors $\nu^i_0$ span all of $\mathbb{R}^2$ due to the angle condition at initial time, we conclude that $\gamma^i(0)=0$ for all $i\in\{1,2,3\}$. 

We now proceed similarly as in the $C^0$-case ~\ref{LopatinskiiShapiro}. 

Testing the motion equation by $\frac{1}{\vert\varphi^i_x\vert}\overline{\langle\gamma^i_{xx},\tau^i_0\rangle}\tau^i_0$, summing and integrating by part gives
\begin{align*}
	0&=-\lambda\sum_{i=1}^{3}\frac{1}{\vert\varphi^i_x\vert}\int_{0}^{\infty}\left\vert\langle\gamma^i_x,\tau^i_0\rangle\right\vert^2\mathrm{d}x-\sum_{i=1}^{3}\frac{1}{\vert\varphi^i_x\vert^5}\int_{0}^{\infty}\left\vert\langle\gamma^i_{xxx},\tau^i_0\rangle\right\vert^2\mathrm{d}x\\
	&+\lambda\sum_{i=1}^{3}\frac{1}{\vert\varphi^i_x\vert}\langle\gamma^i(0),\tau^i_0\rangle\overline{\langle\gamma^i_x(0),\tau^i_0\rangle}+\sum_{i=1}^{3}\frac{1}{\vert\varphi^i_x\vert^5}\langle\gamma^i_{xxx}(0),\tau^i_0\rangle\overline{\langle\gamma^i_{xx}(0),\tau^i_0\rangle}\,.
\end{align*}
The two boundary terms vanish due to $\gamma^i(0)=0$ and the tangential second order condition.
Considering the real part we find $\langle\gamma^i_x(x),\tau^i_0\rangle=0$ for all $x\geq 0$ and hence $\gamma^i_x(x)=0$ for all $x\geq 0$. Differentiating and using the motion equation and $\Re(\lambda)>0$ implies that $\gamma$ is the trivial solution.

In the case  $x=1$ the solution to the ODE can be calculated explicitly and the result 
is trivial. 
\end{proof}

\subsubsection{Contraction estimates}

We now give some remarks on the contraction estimates that are needed to deduce well--posedness of~\eqref{TriodC1} from existence of the linearised system. As the motion equation is unchanged and all boundary terms at $x=1$ are linear, we only provide arguments for the non--linear boundary conditions at $x=0$.

In all the estimates we use the identification of the function spaces presented in~\ref{parabolicspaces}.

\medskip

As the third order conditions in~\eqref{TriodC1} and~\eqref{thetaC0} only differ in their lower order terms, it remains to estimate the expression
\begin{equation}\label{thirdremain}
\frac{\gamma^i_x\vert\gamma^i_{xx}\vert^2}{\vert\gamma^i_x\vert^5}-\frac{\tilde{\gamma}^i_x\vert\tilde{\gamma}^i_{xx}\vert^2}{\vert\tilde{\gamma}^i_x\vert^5}\,
\end{equation}
in the norm of $C^{\frac{1+\alpha}{4}}\left([0,T];(\mathbb{R}^3)^2\right)$.
By adding and subtracting and using
\begin{align*}
\frac{1}{a^5}-\frac{1}{b^5}=\left(\frac{1}{a}-\frac{1}{b}\right)\left(\frac{1}{a^4}+\frac{1}{a^3b}+\frac{1}{a^2b^2}+\frac{1}{ab^3}+\frac{1}{b^4}\right)\,.
\end{align*}
the expression~\eqref{thirdremain} can be written in the following form:
\begin{align*}
\left(\gamma^i_x-\tilde{\gamma}^i_x\right)h(\gamma_x,\gamma_{xx},\tilde{\gamma}_x,\tilde{\gamma}_{xx})+\left(\gamma^i_{xx}-\tilde{\gamma^i_{xx}}\right)g(\gamma_x,\gamma_{xx},\tilde{\gamma}_x,\tilde{\gamma}_{xx})\,.
\end{align*}
Using Lemma~\ref{notprovedbis} the whole term can be estimated by 
$\tilde{C}(M,C) T^\varepsilon\Vert\gamma-\tilde{\gamma}\Vert_{\mathbb{E}_T}$.

Moreover we need to estimate
\begin{align*}
\langle\gamma^i_{xx},\tau^i_0-\tau^i\rangle-\langle\tilde{\gamma}^i_{xx},\tau^i_0-\tilde{\tau}^i\rangle=\langle\gamma^i_{xx}-\tilde{\gamma}^i_{xx},\tau^i_0-\tau^i\rangle+\langle\tilde{\gamma}^i_{xx},\tilde{\tau}^i-\tau^i\rangle\,
\end{align*}
and 
\begin{align*}
&\sum_{i=1}^3\frac{1}{\vert \varphi^i_x\vert^2}
\left\langle \gamma^i_{xx},\nu_0^i\right\rangle 
- \sum_{i=1}^3\frac{1}{\vert \gamma^i_x\vert^2}
\left\langle \gamma^i_{xx},\nu^i\right\rangle -\sum_{i=1}^3\frac{1}{\vert \varphi^i_x\vert^2}
\left\langle \tilde{\gamma}^i_{xx},\nu_0^i\right\rangle 
+ \sum_{i=1}^3\frac{1}{\vert \tilde{\gamma}^i_x\vert^2}
\left\langle \tilde{\gamma}^i_{xx},\tilde{\nu}^i\right\rangle\\
=&\sum_{i=1}^{3}\left(\frac{1}{\vert\varphi^i_x\vert^2}-\frac{1}{\vert\gamma^i_x\vert^2}\right)\langle\gamma^i_{xx}-\tilde{\gamma}^i_{xx},\nu^i_0\rangle+\sum_{i=1}^{3}\frac{1}{\vert\gamma^i_x\vert^2}\langle\gamma^i_{xx}-\tilde{\gamma}^i_{xx},\nu^i_0-\nu^i\rangle\\
+&\sum_{i=1}^{3}\left(\frac{1}{\vert\tilde{\gamma}^i_x\vert^2}-\frac{1}{\vert\gamma^i_x\vert^2}\right)\langle\tilde{\gamma}^i_{xx},\nu^i\rangle+\sum_{i=1}^{3}\frac{1}{\vert\tilde{\gamma}^i_x\vert^2}\langle\tilde{\gamma}^i_{xx},\tilde{\nu}^i-\nu^i\rangle
\end{align*}
in the norm of $C^{\frac{2+\alpha}{4}}\left([0,T];(\mathbb{R}^3)^2\right)$
which can be done as before using~\ref{notprovedbis}.

It remains to study the angle condition.
We let $f:\mathbb{R}^2\to\mathbb{R}^2$ be a smooth function such that $f(p)=\frac{p}{\vert p\vert}$ on $\mathbb{R}^2\setminus B_{C/4}(0)$. As $\gamma^i_x(t), \tilde{\gamma}^i_x(t)$ and $\varphi^i_x$ all lie in $\mathbb{R}^2\setminus B_{C/2}(0)$ for $t\in (0,\tilde{t}(M,C)]$, we can write the expression $b^1(\gamma)-b^1(\tilde{\gamma})$ as
\begin{align*}
&\sum_{i=1}^{3}f(\gamma^i_x)-f(\tilde{\gamma}^i_x)-Df(\varphi^i_x)\gamma^i_x+Df(\varphi^i_x)\tilde{\gamma}^i_x\\
=&\sum_{i=1}^{3}\left( \int_{0}^1 Df(\xi\gamma^i_x+(1-\xi)\tilde{\gamma}^i_x)\,\mathrm{d}\xi-Df(\varphi^i_x)\right) 
(\gamma^i_x-\tilde{\gamma}^i_x)\,.
\end{align*}
The function
$$
t\mapsto g^i(t):=\int_{0}^1 Df(\xi\gamma_x^i(t)+(1-\xi)\tilde{\gamma}^i_x(t))\,\mathrm{d}\xi
$$
is in $C^{\frac{3+\alpha}{4}}\left([0,T];\left(\mathbb{R}^{2\times 2}\right)^2\right)$ and $g(0)=Df(\varphi_x)$ which implies
$$
\sup_{t\in[0,T]}\lVert g^i(t)-g^i(0)\rVert\leq T^{\frac{3+\alpha}{4}} \lVert g^i\rVert_{C^{\frac{3+\alpha}{4}}([0,T])}
$$
for any matrix norm $\lVert\cdot\rVert$. As both $Df$ and $D^2f$ are bounded on  $\overline{B_{M}(0)}$, we have for all $T\in (0,\tilde{t}(M,C)]$ that $\lVert g^i\rVert_{C^{\frac{3+\alpha}{4}}([0,T])}\leq C(M,C)$ . Thus we conclude that
\begin{align*}
&\lVert \left(g^i-g^i(0)\right)\left(\gamma^i_x-\tilde{\gamma}^i_x\right)\rVert_{C^{\frac{3+\alpha}{4}}([0,T])}\\
&\leq \sup_{t\in[0,T]}\lVert g^i(t)-g^i(0)\rVert\lVert\gamma^i_x-\tilde{\gamma}^i_x\rVert_{C^{\frac{3+\alpha}{4}}([0,T])}+\lVert g^i-g^i(0)\rVert_{C^{\frac{3+\alpha}{4}}([0,T])}\sup_{t\in[0,T]}\lvert \gamma^i_x(t)-\tilde{\gamma}^i_x(t)\rvert\\
&\leq C(M,C) T^{\varepsilon}\lVert\gamma-\tilde{\gamma}\rVert_{\mathbb{E}_T}\,
\end{align*}
where we used
$$
\sup_{t\in[0,T]}\lvert\gamma^i_x(t)-\tilde{\gamma}^i_x(t)\rvert \leq \lVert\gamma^i_x-\tilde{\gamma}^i_x\rVert_{C^{\frac{\alpha}{4},\alpha}([0,T]\times[0,1])}\leq T^\varepsilon\lVert \gamma-\tilde{\gamma}\rVert_{\mathbb{E}_T}\,
$$
which follows from Lemma~\ref{notproved}.

As all other arguments are similar to the $C^0$--case we hence
established Theorem 
\ref{geoexistence2}.

\bibliographystyle{amsplain}
\bibliography{elasticflow}

\end{document}